\documentclass{article}
\usepackage[utf8]{inputenc}
\usepackage{graphicx}
\usepackage{amsmath}
\usepackage{amsfonts}
\usepackage{lmodern}
\usepackage[T1]{fontenc}
\usepackage[italian]{babel}
\usepackage{tikz-cd}
\usepackage{yfonts}
\usepackage{amssymb}
\usepackage{mathtools}
\usepackage{array}
\usepackage{wasysym}
\usepackage{amsthm}
\addto\captionsitalian{} 
\addto\captionsitalian{}
\theoremstyle{definition}
\newtheorem{definition}{Definition}[section]
\newtheorem{proposition}{Proposition}[definition]
 
\newtheorem{corollary}{Corollary}[definition] 
\newtheorem{lemma}{Lemma}[definition]
\newtheorem{oss}{Remark}[definition]

\usepackage{enumitem}
\newlist{nicenum}{enumerate}{1}
\setlist[nicenum]{%
  label=\textbf{\arabic*.},%
  left=1em,%
  itemsep=0.5ex,%
  topsep=1ex%
}

\begin{document}
\begin{center}
\begin{Large}
On generalized metric structures
\end{Large}
\end{center}
\begin{center}
Andrea Ricciarini
\end{center}
\vspace{0.2cm}
\begin{center}
\textbf{Abstract}
\end{center}
 Let $M$ be a smooth manifold, let $TM$ be its tangent bundle and $T^{*}M$ its cotangent bundle. This paper investigates integrability conditions for generalized metrics, generalized almost para-complex structures, and generalized Hermitian structures on the generalized tangent bundle of $M$, $E=TM \oplus T^{*}M$. In particular, two notions of integrability are considered: integrability with respect to the Courant bracket and integrability with respect to the bracket induced by an affine connection. We give sufficient criteria that guarantee the integrability for the aforementioned generalized structures, formulated in terms of properties of the associated $2$-form and connection. Extensions to the pseudo-Riemannian setting and consequences for generalized Hermitian and K\"ahler  structures are also discussed.
We also describe relationship between generalized metrics and weak metric structures studied in \cite{12}.
\\ \\ \begin{scriptsize}Keywords: generalized geometry, generalized metrics and connections, generalized K\"ahler structures, weak metric structures. \end{scriptsize}
\\ \begin{scriptsize}2020 MSC:  53C05   53C15   53D18 \end{scriptsize}
\section{Introduction}
Generalized geometry, introduced by Hitchin in \cite{8} and subsequently developed by Gualtieri in \cite{7}, naturally combines the structures of the tangent and cotangent bundles and provides a unified language to study geometric structures that unite symplectic and complex geometry. In this paper we concentrate primarily on three fundamental objects in generalized geometry: generalized metrics, generalized almost para-complex structures, and generalized Hermitian structures. We introduce the concept of integrability of a generalized metric and we analyze their integrability not only with respect to the standard Courant bracket $[\cdot,\cdot]_{C}$ but also with respect to the bracket induced by an affine connection $[\cdot,\cdot]_{\nabla}$.
\\ Chapters $5,6,7,8$ contain the original contributions of the paper, where we introduce sufficient conditions and structural criteria that relate properties of the $2$-form and connection associated to a generalized metric $G$ to the integrability of the structures under consideration, both with respect to $[\cdot,\cdot]_{C}$ and $[\cdot,\cdot]_{\nabla}$.
In particular, the paper is organized as follows.
\\ Chapter $2$ introduces the fundamental notions of generalized geometry, following the basic definitions and principal results given in \cite{7,8,9}.
Chapters $3$ and $4$ develop an in-depth and original study of generalized metrics and generalized connections. After a systematic analysis of the principal properties of the connection induced by a generalized metric, these results are applied in Chapter $5$ to the study of integrability of generalized metrics with respect to both the Courant bracket and the bracket induced by the connection associated with the generalized metric (the main statements appear as Propositions 5.1.1, 5.1.2, 5.2.1, 5.2.2). In Chapter $6$ we consider generalized metrics induced by pseudo-Riemannian metrics: in this context we analyze the principal properties of generalized almost para-complex structures and prove that their integrability does not depend on the signature of the underlying metric; we also study the integrability of the generalized almost complex structure induced by a generalized metric arising from a pseudo-Riemannian metric.
In Chapter $7$ we examine certain weak structures associated to generalized metrics, such as weak nearly Kähler structures. Finally Chapter $8$ is devoted to the original study of integrability for generalized Hermitian structures, both with respect to $[\cdot,\cdot]_{C}$ and $[\cdot,\cdot]_{\nabla}$. We then derive sufficient conditions under which a generalized Hermitian structure is a generalized Kähler structure.
\\ This article is partially based upon the author's master's thesis \cite{11}.

\section{Preliminaries}
Let $M$ be a smooth manifold, let $TM$ be its tangent bundle and $T^{*}M$ its cotangent bundle. In what follows we shall denote the smooth sections of $TM$ by $C^{\infty}(TM)$, of $T^{*}M$ by $C^{\infty}(T^{*}M)$.
\subsection{Geometry of the generalized tangent bundle}
\begin{definition}
\cite{7,8} Let $M$ be a smooth manifold. The \emph{generalized tangent bundle} is the vector bundle over $M$
$$E:=TM \oplus T^{*}M.$$
Sections of $E$, $C^{\infty}(E)$, are written as $X+\xi $ with $X \in C^{\infty}(TM)$, $\xi \in C^{\infty}(T^{*}M)$.
\\ $E$ carries the natural symmetric pairing of neutral signature
$$\langle X+\xi ,Y+\eta \rangle :=\iota _{X}\eta +\iota _{Y}\xi .$$
The projection $\pi:E\rightarrow TM $ is the projection on the first summand. 
\\ The subbundles $TM$ and $T^{*}M$ are isotropic for $\langle \cdot ,\cdot \rangle $.
\end{definition}
\begin{definition}
\cite{7,8} The \emph{Dorfman bracket} is the bilinear operation on $C^{\infty}(E)$, $[\cdot ,\cdot ]_{D}$, defined by:
$$[X+\xi ,Y+\eta ]_{D}:=[X,Y]+\mathcal{L}_{X}\eta -\iota _{Y}d\xi ,$$
where $[X,Y]$ is the Lie bracket and $\mathcal{L}$ the Lie derivative. 
\\ The Dorfman bracket satisfies the Jacobi rule:
$$[e_{1},[e_{2},e_{3}]_{D}]_{D}=[[e_{1},e_{2}]_{D},e_{3}]_{D}+[e_{2},[e_{1},e_{3}]_{D}]_{D}$$
$\forall e_{1},e_{2},e_{3} \in C^{\infty}(E)$. The Dorfman bracket is not skew-symmetric.
\end{definition}
\begin{definition}
\cite{7,8} The \emph{Courant bracket}, $[\cdot ,\cdot ]_{C}$, is the skew-symmetrization of the Dorfman bracket:
$$[e_{1},e_{2}]_{C}:=\frac{1}{2}([e_{1},e_{2}]_{D}-[e_{2},e_{1}]_{D}).$$
It is skew and satisfies the Jacobi identity up to an exact term controlled by the pairing:
$$\sum _{cyc}[[e_{1},e_{2}]_{C},e_{3}]_{C}=\frac{1}{3}d\left(\sum _{cyc} \langle [e_{1},e_{2}]_{C},e_{3}\rangle \right) ,$$
where cyc means cyclic permutations.
\end{definition}
\begin{definition} 
$\cite{1,7}$ Let $b \in \Omega^{2}(TM)$ be a $2$-form. The \emph{$b$-transform} of $E$ is the $\langle \cdot ,\cdot \rangle $-orthogonal automorphism of $E$ defined as:
$$e^{b}(X+\xi )=X+\xi +\iota _{X}b$$
for all $X+\xi \in C^{\infty}(E)$. 
\\ If, in addition, $b$ is $d$-closed, then $e^{b}$ preserves the Courant bracket and is referred to as \emph{$b$-field transform}.
\end{definition}
\begin{definition} 
\cite{7,8} Let $H \in \Omega ^{3}(TM)$ be a given $3$-form. The \emph{H-twisted Dorfman bracket} on $C^{\infty}(E)$, $[\cdot ,\cdot ]_{D}^{H}$, is defined by:
$$[X+\xi ,Y+\eta ]_{D}^{H}=[X,Y]+\mathcal{L}_{X}\eta -\iota _{Y}d\xi +\iota _{X}\iota _{Y}H,$$
for $X+\xi ,Y+\eta \in C^{\infty}(E)$.
\\ The $H$-twisted Courant bracket, $[\cdot ,\cdot ]_{C}^{H}$, is the skew-symmetrization of the $H$-twisted Dorfman bracket:
$$[X+\xi ,Y+\eta ]_{C}^{H}=\frac{1}{2}([X+\xi ,Y+\eta ]_{D}^{H}-[Y+\eta ,X+\xi ]_{D}^{H})=$$
$$=[X,Y]+\frac{1}{2}(\mathcal{L}_{X}\eta -\mathcal{L}_{Y}\xi )-\frac{1}{2}d(\iota _{X}\eta -\iota _{Y}\xi )+\iota _{X}\iota _{Y}H.$$
The twisted brackets preserve the canonical pairing $\langle \cdot ,\cdot \rangle $ and are equivariant under $b$-field transforms.
\end{definition}
\begin{definition}
\cite{9} Given an affine connection $\nabla $ on $TM$, it is possible to define the \emph{connection-induced bracket} on $C^{\infty}(E)$, $[\cdot ,\cdot ]_{\nabla}$, as:
$$[X+\xi ,Y+\eta ]_{\nabla}=[X,Y]+\nabla _{X}\eta -\nabla _{Y}\xi .$$
This bracket is skew-symmetric and it satisfies the Jacobi identity if and only if the curvature of $\nabla $ vanishes.
\end{definition}

\subsection{Generalized complex structures}
\begin{definition}
\cite{7,8} A \emph{generalized almost complex structure à la Hitchin} is an endomorphism $\mathfrak{J}:E \rightarrow E$ such that:
$$ \mathfrak{J}^{2}=-Id, \hspace{0.3cm} \langle \mathfrak{J}v,\mathfrak{J}w \rangle =\langle v,w \rangle \hspace{0.25cm} \forall v,w \in C^{\infty}(E).$$
Equivalently, the $+i$-eigenbundle $V_{i}\subset E\otimes \mathbb{C}$ is a maximal isotropic subbundle. The structure $\mathfrak{J}$ is integrable if and only if $V_{i}$ is closed under the Courant bracket (i.e. $V_{i}$ is involutive). 
\\ In the pure spinor formulation, integrability is equivalent, locally, to the existence of a nonzero pure spinor $\rho$ generating $(detV_{i})^{\perp}$ with $d\rho =0$. A generalized almost complex structure integrable with respect to $[\cdot,\cdot]_{C}$ is called a \emph{generalized complex structure}.
\end{definition}
\begin{definition}
Let $\mathfrak{J}$ be an endomorphism of $E$ and let $[\cdot,\cdot]$ be a bilinear bracket on $C^{\infty}(E)$. We define the \emph{Nijenhuis operator} of $\mathfrak{J}$ with respect to $[\cdot,\cdot]$ as the $\mathbb{R}$-bilinear map:
$$N_{\mathfrak{J}}:C^{\infty}(E)\times C^{\infty}(E) \rightarrow C^{\infty}(E)$$
given by the formula:
$$N_{\mathfrak{J}}(v,w)=[\mathfrak{J}v,\mathfrak{J}w]-\mathfrak{J}[\mathfrak{J}v,w]-\mathfrak{J}[v,\mathfrak{J}w]+\mathfrak{J}^{2}[v,w]$$
for all $v,w \in C^{\infty}(E)$.
\end{definition}
\begin{oss}
$N_{\mathfrak{J}}$ defined above is an operator. Whether one calls it a tensor depends on extra properties of the bracket and of $\mathfrak{J}$.
\end{oss}
\begin{definition}
\cite{9} In the most general sense, a \emph{generalized almost complex structure} on $M$ is an almost complex structure $\mathsf{J}$ on $E$, that is, an endomorphism $\mathsf{J}:E \rightarrow E$ such that $\mathsf{J}^{2}=-Id$.
\end{definition}
\begin{oss}
\cite{4,5} Generalized almost complex structures à la Hitchin are referred to in the literature as \emph{strong}, whereas those just introduced are called \emph{weak}.
\end{oss}
\begin{oss}
In general, one cannot speak of the Nijenhuis tensor with respect to the Courant bracket for weak generalized almost complex structures, since such structures are not a priori compatible with the natural pairing $\langle \cdot ,\cdot \rangle $.
\\ However, one can still speak of the Nijenhuis tensor with respect to the bracket $[\cdot,\cdot]_{\nabla}$. 
\\ For strong generalized almost complex structures, the Nijenhuis operator is indeed a tensor with respect to both the Courant bracket $[\cdot,\cdot]_{C}$ and the bracket $[\cdot,\cdot]_{\nabla}$, because in this case the structures are compatible with the pairing $\langle \cdot ,\cdot \rangle$.
\\ So a weak (or strong) generalized almost complex structure $\mathsf{J}$ is said to be \emph{$[\cdot,\cdot]_{\nabla }$-integrable} if its Nijenhuis tensor with respect to $[\cdot,\cdot]_{\nabla }$, $N_{\mathsf{J}}^{[\cdot,\cdot]_{\nabla }}$, vanishes.
\end{oss}
\begin{proposition}
Let $\mathfrak{J}$ be a strong generalized almost complex structure. Then the following results hold:
\begin{nicenum}[label=\roman*.]
\item The Nijenhuis tensor of $\mathfrak{J}$ with respect  to $[\cdot,\cdot]_{C}$, $N_{\mathfrak{J}}$, vanishes if and only if the $+i$-eigenspace $V_{i}$ of $\mathfrak{J}$ is $[\cdot,\cdot]_{C}$-involutive.
\item The Nijenhuis tensor of $\mathfrak{J}$ with respect  to $[\cdot,\cdot]_{\nabla}$, $N_{\mathfrak{J}}^{[\cdot,\cdot]_{\nabla}}$, vanishes if and only if the $\pm i$-eigenspace $V_{\pm i}$ of $\mathfrak{J}$ are both $[\cdot,\cdot]_{\nabla }$-involutive (the same holds for weak structures).
\end{nicenum}
\end{proposition}
\begin{proof}
First of all, observe that the Courant bracket $[\cdot,\cdot]_{C}$ is compatible with the natural pairing $\langle \cdot , \cdot \rangle$ whereas the bracket $[\cdot,\cdot]_{\nabla }$ is not. Therefore, if $V_{i}$ and $V_{-i}$ denote the orthogonal eigenbundles with respect to $\langle \cdot ,\cdot \rangle $ one can conclude that if $V_{i}$ is $[\cdot,\cdot]_{C}$-involutive, then $V_{-i}$ is also $[\cdot,\cdot]_{C}$-involutive. However, the same conclusion does not hold in general when $V_{i}$ is $[\cdot,\cdot]_{\nabla}$-involutive.
\\ $i$ $"\rightarrow "$ Let $v,w \in C^{\infty}(V_{i})$, and observe that 
$$0=N_{\mathfrak{J}}(v,w)=-2[v,w]_{C}-2i \mathfrak{J}[v,w]_{C}.$$
Hence, we have that $[v,w]_{C}=i[v,w]_{C}$ and so $[v,w]_{C} \in C^{\infty}(V_{i})$.
\\ $i$ $"\leftarrow "$ Since $V_{i}$ is $[\cdot,\cdot]_{C}$-involutive, $V_{-i}$ is also $[\cdot,\cdot]_{C}$-involutive, and this allows us to conclude, through direct computations, that $N_{\mathfrak{J}}$ vanishes.
\\ Analogous arguments hold for $ii$, taking into account the hypothesis and the observations made at the beginning of the proof.
\end{proof}
\hspace{-0.6cm} From now on, we will omit the term 'weak'.
\begin{proposition}
\cite{9} Let $\mathsf{J}=\left[ \begin{matrix} H && \alpha \\ \beta && K \end{matrix}\right]$ be an endomorphism of $E$ where \\ $H:TM\rightarrow TM$, $\alpha :T^*M\rightarrow TM$, $\beta :TM\rightarrow T^*M$, $K:T^*M\rightarrow TM$. Then $\mathsf{J}$ is a generalized almost complex structure if and only if the following conditions hold: 
$$\begin{cases} \beta \alpha =-(I+K^2) \\ \alpha \beta =-(I+H^2) \\ H\alpha =-\alpha K \\ \beta H=-K\beta \end{cases}.$$ 
Moreover $\mathsf{J}$ is invariant with respect to the scalar product $\langle \cdot ,\cdot \rangle $, i.e. it is strong, if and only if, in addition to the above conditions, the following also hold:
$$\begin{cases} K=-H^*\\ \alpha =-\alpha ^* \\ \beta =-\beta ^* \end{cases}$$
where $H^*:T^*M\rightarrow T^*M$ is the dual of $H$ defined by $H^*(\xi )(X)=\xi (H(X))$. The condition $\alpha =-\alpha ^*$ means that $\alpha (\xi )(\eta )=-\alpha (\eta )(\xi )$ for all $\xi ,\eta \in T^*M$ while $\beta = -\beta ^*$ means that $\beta (X)(Y)=-\beta (Y)(X)$ for all $X,Y \in C^{\infty}(TM)$. 
\end{proposition}

\section{Generalized metrics}
\begin{definition}
\cite{7} Let $E$ be the generalized tangent bundle of $M$. A \emph{generalized metric} $V$ on $E$ is a positive-definite subbundle of $E$ with respect to the natural pairing $\langle \cdot ,\cdot \rangle$, of rank equal to $dim_{\mathbb{R}}M$.
\end{definition}
\begin{oss}
Given a generalized metric $V$ on $E$, we define $V^{+}:=V$ and $V^{-}:=V^{\perp}$, the orthogonal complement of $V$ in $E$ with respect to $\langle \cdot , \cdot \rangle $.
\end{oss}
\begin{oss}
A generalized metric $V^{+}$ can also be obtained by means of an endomorphism $G:E \rightarrow E$ defined as $G|_{V^{+}}=Id$ and $G|_{V^{-}}=-Id$. So $G$ satisfies:
\begin{nicenum}[label=\roman*.]
\item $G^{2}=Id.$
\item $\langle Gu,Gv \rangle =\langle u,v \rangle \hspace{0.3cm} \forall u,v \in C^{\infty}(E)$
\item The bilinear form $(u,v)\mapsto \langle Gu,v \rangle $ is positive definite on $V^{+}$ and negative definite on $V^{-}$.
\item Its $\pm 1$-eigenbundles have the same rank.
\end{nicenum}
In this way, $V^{\pm}$ are the $\pm 1$-eigenbundles of $G$, and therefore defining $V^{+}$ is equivalent to specifying $G$.
\end{oss}
\begin{oss}
\cite{1,7} A generalized metric is determined by a pair $(g,b)$ consisting of a Riemannian metric $g$ and a 2-form $b$ on $M$. Indeed, $V^{+}$ is the graph of $g+b$: $$ V^{+}=\{X+(g+b)X|X\in C^{\infty}(TM)\}$$ and the endomorphism $G$ can be written as follows:
$$G=e^{-b}G_{g}e^{b}=e^{-b}\left( 
\begin{matrix} 0 && g^{-1} \\ g && 0 \end{matrix}\right) e^{b} =
\left( \begin{matrix} 
-g^{-1}b && g^{-1} \\ 
g-bg^{-1}b && bg^{-1} 
\end{matrix}\right)$$
where $e^{b}: X+\xi \mapsto X+\xi +\iota _{X}b$ and $G_{g}=\left( \begin{matrix} 0 && g^{-1} \\ g && 0 \end{matrix} \right).$
\end{oss}
\hspace{-0.6cm} In the following, we will refer to a generalized metric indifferently as $(g,b)$, $V^{+}$ or $G$.

\section{Generalized connections}
\begin{definition}
\cite{6} A \emph{generalized connection} $\tilde{\nabla}$ on $E$ is an $\mathbb{R}$-linear map
$$\tilde{\nabla}: \Gamma(E) \rightarrow \Gamma(E^{*}\otimes E)$$ 
satisfying the following conditions:
\begin{nicenum}[label=\roman*.]
\item $\tilde{\nabla}_{v}(fw)=\pi (v)(f)w+f\tilde{\nabla}_{v}w \hspace{0.3cm} \forall v,w \in \Gamma (E), \forall f \in C^{\infty}(TM)$.
\item $\pi (u)\langle v,w \rangle =\langle \tilde{\nabla}_{u}v,w\rangle +\langle v,\tilde{\nabla}_{u}w \rangle $.
\end{nicenum} 
\end{definition}
\begin{lemma}
The generalized connection $\tilde{\nabla}_{X+\xi}(Y+\eta )=\nabla _{X}Y+\nabla _{X} \eta \hspace{0.4cm}$ \\  $\forall X+\xi ,Y+\eta \in C^{\infty}(E)$, induced by an affine connection $\nabla $, is well-defined.
\end{lemma}
\begin{proof}
Condition i. follows from the definition of $\tilde{\nabla}$ in terms of the affine connection $\nabla $. We verify condition ii., or compatibility of $\tilde{\nabla}$ with $\langle \cdot ,\cdot  \rangle$ , by observing that:
$$2X\langle Y+\eta ,Z+\zeta \rangle =X(\eta (Z)+\zeta (Y)),$$
$$2 \langle \nabla _{X}Y+\nabla _{X}\eta ,Z+\zeta \rangle +\langle Y+\eta ,\nabla _{X}Z+\nabla _{X}\zeta \rangle =$$
$$=(\nabla _{X}\eta )(Z)+\zeta ( \nabla _{X}Y)+\eta (\nabla _{X}Z)+(\nabla _{X}\zeta )(Y)=X(\eta (Z)+\zeta (Y)).$$
Thus, we have obtained the required compatibility:
$$X\langle Y+\eta ,Z+\zeta \rangle = \langle \nabla _{X}Y+\nabla _{X}\eta ,Z+\zeta \rangle +\langle Y+\eta ,\nabla _{X}Z+\nabla _{X}\zeta \rangle .$$
\end{proof}
\begin{oss}
\cite{1} Definition 4.1 can be extended to subbundles of $E$. Moreover, let be $v,w \in C^{\infty}(V^{+})$ and $X \in C^{\infty}(TM)$. If we set $u=X^{-}$, i.e., the extension of $X$ to $V^{-}$, $X^{-}=X-gX+bX$, then the following equality holds:
$$X\langle v,w \rangle =\langle [X^{-},v]_{C},w \rangle +\langle v, [X^{-}, w]_{C}\rangle .$$
\end{oss}
\begin{definition}
\cite{1} We define the \emph{generalized connection associated} with a generalized metric $V^{+}$ as: $$\nabla _{X}^{+}v=([X^{-},v]_{C})_{+},$$
where the superscript $\pm $ for a vector field denotes the extension to $V^{\pm}$, \\ $X^{\pm}=X\pm gX+bX$, while the subscript denotes the projection from $E$ to $V^{\pm}$, respectively.
\end{definition}
\begin{oss}
\cite{1} By using the projection $\pi $, it is possible to identify $V^{+}$ with $TM$ and thus obtain the following affine connection on the tangent bundle:
$$ \nabla _{X}^{+}Y=\pi (([X^{-},Y^{+}]_{C})_{+}).$$
Therefore, we have $(\nabla _{X}^{+}Y)^{+}=\nabla _{X}^{+}Y^{+}$.
\end{oss}
\begin{proposition}
\cite{1} The connection $\nabla ^{+}$ on $M$ induced by a generalized metric $V^{+}$ is compatible with the Riemannian metric $g$ associated with $V^{+}$. Moreover, the torsion $3$-tensor is given by:
$$T^{\nabla ^{+}}(X,Y,Z)=g(T^{\nabla ^{+}}(X,Y),Z)=-db(X,Y,Z).$$
\end{proposition}
\begin{oss}
The torsion determined in previous proposition refers (as do all other results) to untwisted Courant bracket. In the case where one considers Courant bracket twisted by a $3$-form $H$, the torsion becomes $T^{\nabla ^{+}}=-db-H$. Moreover, the same reasoning applied to $V^{+}$ also holds for $V^{-}$, yielding a connection $\nabla ^{-}$ on $M$ compatible with $g$ and with torsion $T^{\nabla ^{-}}=db+H$.
\end{oss}
\begin{lemma}
Let $G$ be a generalized metric on $E$ defined by $(g,b)$. Let $\nabla ^{+}$ be the affine connection on $M$ induced by $G$. Then, the following results hold:
\begin{nicenum}[label=\roman*.]
\item If $db=0$ then $(\nabla ^{+}_{X}b)(Y,Z)+(\nabla _{Y}^{+}b)(Z,X)+(\nabla _{Z}^{+}b)(X,Y)=0$.
\item If $\nabla ^{+}b=0$ then $b(T^{\nabla ^{+}}(X,Y),Z) -b(T^{\nabla ^{+}}(Z,X),Y)-b(T^{\nabla ^{+}}(Y,Z),X)=$ \\ $=db(X,Y,Z).$
\end{nicenum}
\end{lemma}
\begin{proof}
The proof follows from direct computations, observing that $T^{\nabla ^{+}}=-db$ and using the following formula:
$$db(X,Y,Z)=\sum _{cyc}((\nabla _{X}^{+}b)(Y,Z)-b(T^{\nabla ^{+}}(X,Y),Z)).$$
\end{proof}
\begin{lemma}
Let $G$ be a generalized metric on $E$ defined by $(g,b)$. Let $\nabla ^{+}$ be the affine connection on $M$ induced by $G$ and $\tilde{\nabla}$ the generalized connection induced by $\nabla^{+}$. Then $\nabla ^{+}b=0 \iff \tilde{\nabla}G=0$.
\end{lemma}
\begin{proof}
"$\rightarrow $" We have already noted that $\nabla ^{+}g=0$, so:
$$\tilde{\nabla} _{X+\xi }(G(Y+\eta ))=\nabla _{X}^{+}((-g^{1}b)(Y)+g^{-1}\eta )+\nabla _{X}^{+}((g-bg^{-1}b)(X)+bg^{-1}(\eta ))=$$
$$=-g^{-1}b\nabla _{X}^{+}Y+g^{-1}\nabla _{X}^{+}\eta +(g-bg^{-1}b)\nabla _{X}^{+}Y+bg^{-1}\nabla _{X}^{+}\eta =G(\tilde {\nabla} _{X+\xi}(Y+\eta)).$$
"$\leftarrow$" If $\tilde{\nabla}G=0$, then $g^{-1}\nabla _{X} ^{+}(bY)=g^{-1}b(\nabla _{X}^{+}Y)$   $\forall X,Y \in C^{\infty}(M)$. So $\nabla ^{+}b$ vanishes.
\end{proof}
\begin{definition}
\cite{6} The \emph{torsion} $T^{\tilde{\nabla}} \in C^{\infty}(\bigwedge ^{2}E \otimes E)$ of a generalized connection $\tilde{\nabla}$ on $E$ with respect to the Courant bracket, is defined as:
$$T^{\tilde{\nabla}}(v,w,u)=\langle \tilde{\nabla} _{v}w-\tilde{\nabla} _{w}v-[v,w]_{C},u\rangle + \frac{1}{2} (\langle \tilde{\nabla}_{u}v,w \rangle -\langle \tilde{\nabla}_{u}w,v\rangle ).$$
The \emph{torsion} $T^{\tilde{\nabla}}_{[\cdot,\cdot]_{\nabla}}$ of a generalized connection $\tilde{\nabla}$ on $E$ with respect to $[\cdot,\cdot]_{\nabla}$, where $\nabla$ is an affine connection on $TM$, is defined as:
$$T_{[\cdot,\cdot]_{\nabla}}^{\tilde{\nabla}}(v,w)=\tilde{\nabla}_{v}w-\tilde{\nabla}_{w}v-[v,w]_{\nabla}.$$
\end{definition}
\begin{oss}
\cite{3}The torsion $T^{\tilde{\nabla}}$ of a generalized connection $\tilde{\nabla}$ can be expressed equivalently in terms of $[\cdot,\cdot]_{D}$ as follows:
$$T^{\tilde{\nabla}}(v,w,u)=\langle \tilde{\nabla}_{v}w-\tilde{\nabla}_{w} v-[v,w]_{D},u\rangle +\langle \tilde{\nabla}_{u}v,w\rangle ;$$
$$T^{\tilde{\nabla}}(v,w)=\tilde{\nabla}_{v}w-\tilde{\nabla}_{w}v-[v,w]_{D}+(\tilde{\nabla}v)^{*}w$$
where $(\tilde{\nabla}v)^{*}$ is defined by: $\langle u,(\tilde{\nabla}v)^{*}w\rangle =\langle \tilde{\nabla}_{u}v,w \rangle$.
\end{oss}
\begin{lemma}
Let $\nabla$ be a torsion-free affine connection on $M$ and let $\tilde{\nabla}$ be the induced generalized connection. Then $T^{\tilde{\nabla}}=T_{[\cdot,\cdot]_{\nabla}}^{\tilde{\nabla}}=0$.
\end{lemma}
\begin{proof}
We begin by computing the torsion $T^{\tilde{\nabla}}$; let $v=X+\xi$, $w=Y+\eta$, $u=Z+\zeta $ $\in C^{\infty}(E)$:
$$T^{\tilde{\nabla}}(v,w,u)=\langle \tilde{\nabla}_{v}w-\tilde{\nabla}_{w}v-[v,w]_{D},u \rangle + \langle u,(\tilde{\nabla} v)^{*}w\rangle =$$
$$=\langle \nabla _{X}\eta -\nabla _{Y}\xi +\iota _{Y}d\xi -\mathcal{L}_{X}\eta ,Z+\zeta \rangle +\langle \tilde{\nabla}_{u}v,w\rangle =$$
$$=\frac{1}{2}((\nabla _{X}\eta )(Z)-(\nabla _{Y}\xi )(Z)+d\xi (Y,Z)-(\mathcal{L}_{X}\eta )(Z)+\eta (\nabla _{Z}X)+(\nabla _{Z}\xi )(Y))=$$
$$= \frac{1}{2}(\nabla _{X}(\eta (Z))-\eta (\nabla _{X}Z)-\nabla _{Y}(\xi (Z))+\xi (\nabla _{Y}Z)+\nabla _{Y}(\xi (Z))+$$
$$-\nabla _{Z}(\xi (Y))-\xi ([Y,Z])-d\eta (X,Z)-d(\eta (X))(Z)+\eta (\nabla _{Z}X)+(\nabla _{Z}\xi )(Y))=$$
$$= \frac{1}{2}((\nabla _{X}(\eta (Z))-\eta (\nabla _{X}Z)-\nabla _{Y}(\xi (Z))+\xi (\nabla _{Y}Z)+\nabla _{Y}(\xi (Z))+$$
$$-\nabla _{Z}(\xi (Y))-\xi ([Y,Z])-\nabla _{X}(\eta (Z))+\nabla _{Z}(\eta (X))+\eta ([X,Z])-\nabla _{Z}(\eta (X))+$$
$$+\eta (\nabla _{Z}X)+(\nabla _{Z}\xi )(Y))=$$
$$=\frac{1}{2}(-\eta (\nabla _{X}Z)-\nabla _{Y}(\xi (Z))+\xi (\nabla _{Y}Z)+\nabla_{Y}(\xi (Z))-\nabla _{Z}(\xi (Y))+$$
$$-\xi (\nabla _{Y}Z)+\xi (\nabla _{Z}Y)+\eta ([X,Z])+\eta (\nabla _{Z}X)+(\nabla _{Z}\xi )(Y))=0.$$
We now compute the other torsion $T_{[\cdot,\cdot]_{\nabla}}^{\tilde{\nabla}}$:
$$T_{[\cdot,\cdot]_{\nabla}}^{\tilde{\nabla}}(X+\xi ,Y+\eta )=\nabla _{X}Y+\nabla _{X}\eta -\nabla _{Y}X-\nabla _{Y}\xi -[X+\xi ,Y+\eta ]_{\nabla}=T^{\nabla}(X,Y)=0.$$
\end{proof}
\begin{lemma}
Let $\tilde{\nabla}$ be a generalized connection on $E$ and let $\nabla $ be an affine connection on $TM$. The torsions $T^{\tilde{\nabla}}$ and $T^{\tilde{\nabla}}_{[\cdot,\cdot]_{\nabla}}$ coincide if and only if:
$$\frac{1}{2}(\eta (T^{\nabla}(X,Z)+\nabla _{Z}X)-\xi (T^{\nabla}(Y,Z)+\nabla _{Z}Y)+Z(\xi (Y)))=\langle \tilde{\nabla}_{Z+\zeta}(X+ \xi ),Y+\eta \rangle $$
$$\forall v=X+\xi ,w=Y+\eta ,u=Z+\zeta \in C^{\infty}(E).$$
\end{lemma}
\begin{proof}
Since a generalized connection is compatible with $\langle \cdot ,\cdot \rangle $, we can consider:
$$T_{[\cdot,\cdot]_{\nabla}}^{\tilde{\nabla}}(v,w,u)=\langle \tilde{\nabla}_{v}w-\tilde{\nabla}_{w}v-[v,w]_{\nabla},u\rangle .$$
We can thus say that:
$$T_{[\cdot,\cdot]_{\nabla}}^{\tilde{\nabla}}(v,w,u)=T^{\tilde{\nabla}}(v,w,u)\iff \langle [v,w]_{D}-[v,w]_{\nabla},u\rangle = \langle \tilde{\nabla}_{u}v,w\rangle .$$
Note now that, since $(\mathcal{L}_{X}\eta )(Y)=X(\eta (Y))-\eta ([X,Y])$ and $(\nabla _{X}\eta )(Y)=$ \\ $=X(\eta (Y))-\eta (\nabla _{X}Y)$, we obtain:
$$\langle [v,w]_{D}-[v,w]_{\nabla},u\rangle =\langle \mathcal{L}_{X}\eta -\mathcal{L}_{Y}\xi +d(\xi (Y))-\nabla _{X}\eta +\nabla _{Y}\xi , Z+\zeta \rangle=$$
$$=\frac{1}{2}(\eta (T^{\nabla}(X,Z)+\nabla _{Z}X)-\xi (T^{\nabla}(Y,Z)+\nabla _{Z}Y)+Z(\xi (Y)).$$
From the above, the thesis follows.
\end{proof}
\begin{corollary} 
Let $\tilde{\nabla}$ be a generalized connection induced by an affine connection $\nabla $. The torsions $T^{\tilde{\nabla}}$ and $T_{[\cdot,\cdot]_{\nabla}}^{\tilde{\nabla}}$ are equal if and only if they both vanish (i.e. if and only if $\nabla$ is torsion-free).
\end{corollary}
\begin{proof}
Since $\tilde{\nabla}_{X+\xi }(Y+\eta )=\nabla _{X}Y+\nabla _{X}\eta $ we obtain:
$$\langle \tilde{\nabla}_{Z+\zeta}(X+ \xi ),Y+\eta \rangle =\frac{1}{2}(\eta (\nabla _{Z}X)+(\nabla _{Z}\xi )(Y)).$$
Substituting into previous lemma, $T^{\tilde{\nabla}}$ and $T_{[\cdot,\cdot]_{\nabla}}^{\tilde{\nabla}}$ are the same if and only if:
$$\eta (T^{\nabla}(X,Z))-\xi (T^{\nabla}(Y,Z))=0 \hspace{0.3cm} \forall v=X+\xi ,w=Y+\eta ,u=Z+\zeta \in C^{\infty}(E).$$ 
This equality holds for all $\xi $ and for all $\eta$ in $C^{\infty}(T^{*}M)$ and so $T^{\nabla}=0$. Therefore, by Lemma 4.3.1 $T^{\tilde{\nabla}}$ and $T_{[\cdot,\cdot]_{\nabla}}^{\tilde{\nabla}}$ are equal if and only if $\nabla $ is torsion-free.
\end{proof}

\section{Integrability of generalized metrics}
\subsection{Courant integrability}
\begin{definition}
A generalized metric $G$ on $E$ is said to be \emph{$[\cdot,\cdot]_{C}$-integrable} if its $+1$-eigenbundle, $V^{+}$, is involutive with respect to $[\cdot,\cdot]_{C}$.
\end{definition}
\begin{lemma}
A generalized metric is $[\cdot,\cdot]_{C}$-integrable if and only if its Nijenhuis tensor with respect to the Courant bracket, $N_{G}$, vanishes:
$$ N_{G}(v,w)=[Gv,Gw]_{C}-G[Gv,w]_{C}-G[v,Gw]_{C}+[v,w]_{C}=0 \hspace{0.2cm} \forall v,w \in C^{\infty}(E).$$
\end{lemma}
\begin{proof}
The proof follows the same arguments as those used in Proposition 2.9.1.
\end{proof}
\begin{oss}
A generalized metric $G$ always induces a connection $\nabla ^{+}$ on $M$. Thus, one can define the generalized connection $\tilde{\nabla}$ induced by $\nabla ^{+}$ and rewrite the Nijenhuis tensor with respect to the Courant bracket as follows:
$$N_{G}(v,w)=\tilde{\nabla}_{Gv}(Gw)-\tilde{\nabla}_{Gw}(Gv)-T^{\tilde{\nabla}}(Gv,Gw)+(\tilde{\nabla}Gv)^{*}(Gw)+$$
$$-G(\tilde{\nabla}_{Gv}w-\tilde{\nabla}_{w}(Gv)-T^{\tilde{\nabla}}(Gv,w)+(\tilde{\nabla}Gv)^{*}w)-G(\tilde{\nabla}_{v}(Gw)-\tilde{\nabla}_{Gw}V+$$
$$-T^{\tilde{\nabla}}(v,Gw)+(\tilde{\nabla}v)^{*}Gw)+\tilde{\nabla}_{v}w-\tilde{\nabla}_{v}w-T^{\tilde{\nabla}}(v,w)+(\tilde{\nabla}v)^{*}w.$$
\end{oss}
\begin{lemma}
Let $G_{g}$ be the generalized metric induced by a Riemannian metric $g$ and with $b=0$. Let $\tilde{\nabla}$ be the generalized connection induced by $\nabla ^{+}$ obtained from $G_{g}$. If:
$$ \langle \tilde{\nabla}_{u}v,w\rangle =\langle \tilde{\nabla}_{G_{g}u}v,G_{g}w\rangle \hspace{0.2cm} \forall u,v,w \in C^{\infty}(E)$$
then $G_{g}$ is $[\cdot,\cdot]_{C}$-integrable.
\end{lemma}
\begin{proof} Since $b=0$ and by Lemma 3.4.2, Lemma 3.2.2 follows that $T^{\tilde{\nabla}}=0$ and $\tilde{\nabla}G_{g}=0$. Thus, the Nijenhuis tensor of $G_{g}$ with respect to $[\cdot,\cdot]_{C}$ becomes:
$$N_{G_{g}}(v,w)=(\tilde{\nabla}Gv)^{*}(Gw)-G((\tilde{\nabla}Gv)^{*}w)-G((\tilde{\nabla}v)^{*}Gw)+(\tilde{\nabla}v)^{*}w.$$
As $G$ is orthogonal with respect to $\langle \cdot ,\cdot \rangle $ from the hypothesis it follows that:
$$\langle N_{G_{g}}(v,w),u\rangle =2(\langle \tilde{\nabla}_{u}v,w\rangle -\langle \tilde{\nabla}_{G_{g}u}v,G_{g}w\rangle )=0.$$
\end{proof}
\begin{lemma}
Let $G_{g}$ be the generalized metric induced by a Riemannian metric $g$ and with $b=0$. Then its $+1$-eigenbundle, $V^{+}$, is $[\cdot,\cdot]_{C}$-involutive if $g(\nabla _{Z}X,Y)=g(\nabla _{Z}Y,X).$
\end{lemma}
\begin{proof}
$$ [X+gX,Y+gY]_{C}=[X,Y]+\mathcal{L}_{X}(gY)-\mathcal{L}_{Y}(gX)=[X,Y]+\iota _{X}(d(g(Y))-\iota _{Y}(d(g(X)))$$
and from
$$\iota _{X}(d(g(Y))-\iota _{Y}(d(g(X)))(Z)=g([X,Y],Z)+g(\nabla _{Z}X,Y)-g(\nabla _{Z}Y,X)$$ 
by using $g(\nabla _{Z}X,Y)=g(\nabla _{Z}Y,X)$ we obtain the involutivity of $V^{+}$.
\end{proof}
\begin{corollary}
Let $G_{g}$ be the generalized metric defined by $(g,0)$ where $g$ is a Riemannian metric. The following condition is sufficient for the $[\cdot,\cdot]_{C}$-integrability of $G_{g}$:
%\begin{nicenum}[label=\roman*.]
%\item 
$$\beta (Z):=g(\nabla _{Z}^{+}X,Y)-g(\nabla _{Z}^{+}Y,X)=0 \hspace{0.3cm} \forall Z \in C^{\infty}(TM).$$
%\item $\nabla _{Z}^{+}$ is $g$-self-adjoint $\forall Z \in C^{\infty}(TM)$, that is $g(\nabla _{Z}Y,X)=g(\nabla _{Z}X,Y)$ \\ $\forall X,Y,Z \in C^{\infty}(TM)$.
%\end{nicenum}
\end{corollary}
\hspace{-0.6cm} We now study the relation between the Nijenhuis tensor with respect to $[\cdot,\cdot]_{C}$ and the Nijenhuis operator with respect to $[\cdot,\cdot]_{D}$ for $G$. Note that $N_{G}^{[\cdot,\cdot]_{D}}$ is not a priori skew-symmetric, indeed:
$$N_{G}^{[\cdot,\cdot]_{D}}(v,w)=-N_{G}^{[\cdot,\cdot]_{D}}(v,w)+2(d\langle v,w \rangle -G(d\langle Gv,w\rangle )).$$
Moreover, knowing that $[v,w]_{C}=\frac{1}{2}([v,w]_{D}-[w,v]_{D})$, we obtain that:
$$N_{G}(v,w)=\frac{1}{2}(N_{G}^{[\cdot,\cdot]_{D}}(v,w)-N_{G}^{[\cdot,\cdot]_{D}}(w,v)).$$
\begin{lemma}
Let $G$ be a generalized metric on $E$. $G$ is $[\cdot,\cdot]_{C}$-integrable if and only if $N_{G}^{[\cdot,\cdot]_{D}}$ is symmetric.
\end{lemma}
\begin{proof}
$G$ is $[\cdot,\cdot]_{C}$-integrable $\iff N_{G}=0 \iff N_{G}^{[\cdot,\cdot]_{D}}(v,w)=N_{G}^{[\cdot,\cdot]_{D}}(w,v) \hspace{0.2cm}$ \\ $ \forall v,w \in C^{\infty}(E).$
\end{proof}
\begin{proposition}
Let $G$ be a generalized metric on $E$ defined by $(g,b)$ and let $G_{g}$ be the generalized metric induced by $(g,0)$. If $db=0$ and $G_{g}$ is $[\cdot,\cdot]_{C}$-integrable then $G$ is $[\cdot,\cdot]_{C}$-integrable.
\end{proposition}
\begin{proof}
We observe that $G=e^{-b}G_{g}e^{b}$. Since $db=0$, $e^{b}$ and $e^{-b}$ are $b$-field transforms that preserve the integrability of $G_{g}$. So $G$ is $[\cdot,\cdot]_{C}$-integrable.
\end{proof}
\begin{lemma}
Let $G$ be a generalized metric on $E$ and let $(g,b,\nabla ^{+})$ be the induced structures. Let $\tilde{\nabla}$ be the generalized connection induced by $\nabla ^{+}$. If $\nabla ^{+}b=0$ and 
$$2(\langle \tilde{\nabla}_{u}v,w\rangle -\langle \tilde{\nabla}_{Gu}v,Gw\rangle )-T^{\tilde{\nabla}}(Gv,Gw,u)+ T^{\tilde{\nabla}}(Gv,w,Gu)+$$
$$+ T^{\tilde{\nabla}}(v,Gw,Gu)- T^{\tilde{\nabla}}(v,w,u)=0 \hspace{0.35cm} \forall u,v,w \in C^{\infty}(E)$$ then $G$ is $[\cdot,\cdot]_{C}$-integrable.
\end{lemma}
\begin{proof}
If $\nabla ^{+}b=0$, then $\tilde{\nabla}G=0.$ Thus, from the formula in Remark 5.1.1, by "contracting" it with $\langle \cdot ,\cdot \rangle$, the thesis follows.
\end{proof}
\begin{proposition}
Let $G$ be a generalized metric on $E$ and $(g,b,\nabla ^{+})$ the induced structures. Let $\tilde{\nabla}$ be the generalized connection induced by $\nabla ^{+}$. If $\nabla ^{+}b=0$, $db=0$ and $\langle \tilde{\nabla}_{u}v,w\rangle =\langle \tilde{\nabla}_{Gu}v,Gw\rangle$ then $G$ is $[\cdot,\cdot]_{C}$-integrable.
\end{proposition}
\begin{proof}
If $db=0$, then $T^{\nabla ^{+}}=0$ and by Lemma 4.3.2, $T^{\tilde{\nabla}}=0$. Substituting into previous lemma, the thesis follows.
\end{proof}

\subsection{$[\cdot,\cdot]_{\nabla^{+}}$-integrability}
\begin{definition}
A generalized metric $G$ is said to be \emph{$[\cdot,\cdot]_{\nabla ^{+}}$-integrable}
% if $V^{+}$ is involutive with respect to $[\cdot,\cdot]_{\nabla ^{+}}$ where $\nabla ^{+}$ is the affine connection on $M$ induced by $G$. Equivalently, $G$ is $[\cdot,\cdot]_{\nabla ^{+}}$-integrable 
if its Nijenhuis tensor with respect to the bracket induced by the connection $\nabla ^{+}$ arising from $G$, $N_{G}^{[\cdot,\cdot]_{\nabla ^{+}}}$, vanishes.
\end{definition}
\begin{oss}
By using the definition of the Nijenhuis tensor, we have that $G$ is $[\cdot,\cdot]_{\nabla^{+}}$-integrable if and only if:
$$ N_{G}^{[\cdot,\cdot]_{\nabla ^{+}}}(v,w)=[Gv,Gw]_{\nabla ^{+}}-G[Gv,w]_{\nabla ^{+}}-G[v,Gw]_{\nabla ^{+}}+[v,w]_{\nabla ^{+}}=$$
$$=\tilde{\nabla}_{Gv}(Gw)-\tilde{\nabla}_{Gw}(Gv)-T^{\tilde{\nabla}}_{[\cdot,\cdot]_{\nabla ^{+}}}(Gv,Gw)
-G(\tilde{\nabla}_{Gv}w-\tilde{\nabla}_{w}(Gv)+$$
$$-T^{\tilde{\nabla}}_{[\cdot,\cdot]_{\nabla ^{+}}}(Gv,w))-G(\tilde{\nabla}_{v}(Gw)-\tilde{\nabla}_{Gw}V
-T^{\tilde{\nabla}}_{[\cdot,\cdot]_{\nabla ^{+}}}(v,Gw))+\tilde{\nabla}_{v}w+$$
$$-\tilde{\nabla}_{v}w-T^{\tilde{\nabla}}_{[\cdot,\cdot]_{\nabla ^{+}}}(v,w)=0 \hspace{0.3cm} \forall v,w \in C^{\infty}(E)$$
where $\tilde{\nabla}$ is the generalized connection induced by $\nabla ^{+}$ arising from $G$.
\end{oss}
\begin{lemma}
A generalized metric $G$ is $[\cdot,\cdot]_{\nabla ^{+}}$-integrable if and only if both of its eigenspaces $V_{\pm}$ are involutive.
\end{lemma}
\begin{proof}
The proof follows the same arguments as those used in Proposition 2.9.1.
\end{proof}
\begin{lemma}
Let $G_{g}$ be a generalized metric and $(g,b=0,\nabla ^{+})$ the induced structures. $G_{g}$ is $[\cdot,\cdot]_{\nabla^{+}}$-integrable.
\end{lemma}
\begin{proof}
Since $b=0$ we obtain $T^{\nabla ^{+}}=0$ and so $T_{[\cdot,\cdot]_{\nabla ^{+}}}^{\tilde{\nabla}}=0$. Furthermore by Lemma 4.2.2 we obtain $\tilde{\nabla}G=0$ and computing the Nijenhuis tensor we have $N_{G_{g}}^{[\cdot,\cdot]_{\nabla^{+}}}=0.$
\end{proof}
\begin{lemma}
Let $G$ be a generalized metric on $E$ and $(g,b,\nabla^{+})$ the induced structures, or $G=e^{-b}G_{g}e^{b}$. Then $G_{g}$ is $[\cdot,\cdot]_{\nabla^{+}}$-integrable $\iff$ $db=0$.
\end{lemma}
\begin{proof}
$"\leftarrow "$ We already now that $\nabla ^{+}g=0$, so:
$$ N_{G}^{[\cdot,\cdot]_{\nabla ^{+}}}(v,w)=-T^{\tilde{\nabla}}_{[\cdot,\cdot]_{\nabla ^{+}}}(Gv,Gw)+G(T^{\tilde{\nabla}}_{[\cdot,\cdot]_{\nabla ^{+}}}(Gv,w))+$$
$$+G(T^{\tilde{\nabla}}_{[\cdot,\cdot]_{\nabla ^{+}}}(v,Gw))-T^{\tilde{\nabla}}_{[\cdot,\cdot]_{\nabla ^{+}}}(v,w)=0 \hspace{0.3cm} \forall v,w \in C^{\infty}(E).$$
However, since $db=0$, it follows that $N_{G_{g}}^{[\cdot,\cdot]_{\nabla^{+}}}(v,w)=0 \hspace{0.2cm} \forall v,w \in C^{\infty}(E).$
\\ $"\rightarrow "$. By computing the Nijenhuis tensor for $G$ we obtain that:
$$0=N_{G_{g}}^{[\cdot,\cdot]_{\nabla ^{+}}}(X+\xi ,Y+\eta )=T^{\nabla ^{+}}(g^{-1}\xi ,Y)+T^{\nabla ^{+}}(X,g^{-1}\eta ) \hspace{0.3cm} \forall v,w \in C^{\infty}(E)$$
and so $T^{\nabla^{+}}=-db=0$.
\end{proof}
\begin{proposition}
Let $G$ be a generalized metric on $E$ and let $(g,b,\nabla ^{+})$ be the induced structures. If $G_{g}$ is $[\cdot,\cdot]_{\nabla ^{+}}$-integrable and $\nabla ^{+}b=0$, then $G$ is $[\cdot,\cdot]_{\nabla ^{+}}$-integrable.
\end{proposition}
\begin{proof}
We observe that $G=e^{-b}G_{g}e^{b}$. Since $\nabla ^{+}b=0$, we have that $e^{-b}$ and $e^{b}$ are $b$-fields for $[\cdot,\cdot]_{\nabla ^{+}}$. So we obtain that $G$ is $[\cdot,\cdot]_{\nabla ^{+}}$-integrable.
\end{proof}
\begin{oss}
If $b$ is $d$-closed, it defines a cohomology class $[b]$. Hence, $\nabla^{+}$, when $db=0$, is the Levi-Civita connection and does not depend on the chosen representative within the cohomology class $[b]$. Indeed, if we consider $a=b+d\alpha$ we obtain $T^{\nabla ^{+}}=-da=-db-d^{2}\alpha =0.$
\end{oss}
\begin{proposition} 
Let $V^{+}=\{X+gX+bX|X\in C^{\infty}(TM)\}$ and \\ $V^{+}_{d\alpha}=\{X+gX+bX+d\alpha X |X \in C^{\infty}(TM)\}$ be two generalized metrics on $E$ that differ by an exact term. Then the induced connections $\nabla ^{+}$ and $\nabla ^{+,d\alpha}$ coincide.
\end{proposition}
\begin{proof}
If $db=0$, previous remark yields the statement.
Otherwise, if $db$ is nonzero we observe that
$$\nabla _{X}^{+}Y=\pi _{1}(\nabla _{X}^{+}Y^{+})=\pi _{1}(\nabla _{X}^{+}(Y+gY+bY))$$
where $\pi _{1} :V^{+}\rightarrow TM$. The same for the other connection:
$$\nabla _{X}^{+,d\alpha }Y=\pi _{2}(\nabla _{X}^{+,d\alpha }Y^{+})=\pi _{2}(\nabla _{X}^{+,d\alpha }(Y+gY+bY+d\alpha Y))$$
where $\pi _{2}:V_{d\alpha }^{+}\rightarrow TM$.
Thus, performing the computations, we obtain:
$$\nabla _{X}^{+}(Y+gY+bY)=([X^{-},Y^{+}]_{C})_{+}=([X-gX+bX,Y+gY+bY]_{C})_{+}=$$
$$=([X,Y]+\mathcal{L}_{X}(gY+bY)-\mathcal{L}_{Y}(-gX+bX)-\frac{1}{2}d(\iota_{X}(gY+bY)-\iota_{Y}(-gX+bX))_{+}.$$
Knowing that the projection is $(X+\xi )_{+}=\frac{1}{2}(Id+G)(X+\xi )$ it follows that:
$$ \nabla _{X}^{+}Y=\frac{1}{2}([X,Y]-g^{-1}b([X,Y])+g^{-1}(\mathcal{L}_{X}(gY+bY)-\mathcal{L}_{Y}(-gX+bX)+$$
$$-\frac{1}{2}d(\iota _{X}(gY+bY)-\iota _{Y}(-gX+bX)))).$$
Proceeding analogously for the second connection, we obtain:
$$\nabla _{X}^{+,d\alpha }Y=\frac{1}{2}([X,Y]-g^{-1}b([X,Y])-g^{-1}d\alpha ([X,Y])+g^{-1}(\mathcal{L}_{X}(gY+bY)+\mathcal{L}_{X}(d\alpha Y)+$$
$$-\mathcal{L}_{Y}(-gX+bX)-\mathcal{L}_{Y}(d\alpha X)-\frac{1}{2}d(\iota _{X}(gY+bY)-\iota _{Y}(-gX+bX))+$$
$$-\frac{1}{2}d(\iota _{X}(d\alpha Y)-\iota _{Y}(d\alpha X)))).$$
So the following relation holds:
$$\nabla _{X}^{+}Y-\nabla _{X}^{+,d\alpha }Y=\frac{1}{2}g^{-1}(d\alpha ([X,Y])-\mathcal{L}_{X}(\iota _{Y}d\alpha )+\mathcal{L}_{Y}(\iota _{X}d\alpha )+d(\iota _{X}\iota _{Y}d\alpha )).$$
We note, however, that the following properties of the Lie derivative hold
$$-\mathcal{L}_{X}(\iota _{Y}d\alpha )=-\iota _{Y}\mathcal{L}_{X}d\alpha -\iota _{[X,Y]}d\alpha ,$$
$$-\iota _{Y}\mathcal{L}_{X}d\alpha =-\iota _{Y} (d\iota _{X}d\alpha ),$$
$$\mathcal{L}_{Y}(\iota _{X}d\alpha )=d(\iota _{Y}\iota _{X}d\alpha )+\iota _{Y}d(\iota _{X}d\alpha ).$$
thus, by substituting, we obtain:
$$\nabla _{X}^{+}Y-\nabla _{X}^{+,d\alpha }Y=0.$$
\end{proof}
\hspace{-0.58cm} We recall the following result, for completeness we insert a proof.
\begin{proposition}
Two affine connections $\nabla ^{1}$ and $\nabla ^{2}$ on a manifold $M$, compatible with the same Riemannian metric $g$ and having the same torsion $T^{\nabla^{1}}=T^{\nabla ^{2}}$, are equal.
\end{proposition}
\begin{proof}
Since the connection $\nabla ^{1}$ is compatible with $g$, we observe that:
$$Xg(Y,Z)+Yg(X,Z)-Zg(X,Y)=g(\nabla _{X}^{1}Y,Z)+g(Y,\nabla _{X}^{1}Z)+g(\nabla _{Y}^{1}X,Z)+$$
$$+g(X,\nabla _{Y}^{1}Z)-g(\nabla _{Z}^{1}X,Y)-g(X,\nabla _{Z}^{1}Y)=$$
$$=2g(\nabla _{X}^{1}Y,Z)-g(Z,[X,Y])-g(Z,T^{\nabla ^{1}}(X,Y))+g(Y,T^{\nabla ^{1}}(X,Z))+g(Y,[X,Z])+$$
$$+g(X,T^{\nabla ^{1}}(Y,Z))+g(X,[Y,Z]).$$
So the following equality holds:
$$g(\nabla _{X}^{1}Y,Z)=\frac{1}{2}\{Xg(Y,Z)+Yg(X,Z)-Zg(X,Y)+g(Z,[X,Y])-g(Y,[X,Z])+$$
$$-g(X,[Y,Z])+g(Z,T^{\nabla ^{1}}(X,Y))-g(Y,T^{\nabla ^{1}}(X,Z))-g(X,T^{\nabla ^{1}}(Y,Z))\}.$$
Analogously, the same is done for $\nabla ^{2}$. From the assumption on the torsion, the statement follows.
\end{proof}
\begin{oss}
Previous theorem confirms the result established in Proposition 5.2.1. Moreover, this allows us to conclude that the two connections $\nabla ^{+}$ and $\nabla ^{-}$ induced by a generalized metric, differ only in the sign of the torsion terms in the equality from previous proof.
\end{oss}

\section{Generalized metrics with pseudo-Riemannian metric}
\begin{definition}
$\cite{2}$ A \emph{generalized almost product structure} $\mathfrak{P}$ is an endomorphism of $E$ such that $\mathfrak{P}^{2}=Id$ and $\mathfrak{P}$ is orthogonal with respect to $\langle \cdot ,\cdot \rangle$, or $$\langle \mathfrak{P}v,\mathfrak{P}w\rangle=\langle v,w\rangle \hspace{0.2cm} \forall v,w \in C^{\infty}(E).$$
\end{definition}
\begin{definition}
A \emph{generalized almost para-complex structure}, $\mathfrak{P}$, is a generalized almost product structure such that the two eigenbundles associated to $+1$ and $-1$ have the same rank.
\end{definition}
\begin{oss}
Let $G=e^{-b}G_{g}e^{b}$ be an endomorphism of $E$ induced by a 2-form $b$ and by a pseudo-Riemannian metric $g$. Then $G^{2}=Id$ and $G$ is orthogonal with respect to $\langle \cdot ,\cdot \rangle $ and so it is a generalized almost product structure. In other words, if we consider any generalized metric, it represents a special case of a generalized almost product structure. Moreover, when considering generalized metrics $G$ induced by  pseudo-Riemannian metrics $g$, we again obtain generalized almost product structures.
\end{oss}
\begin{oss}
The main difference between a generalized almost para-complex structure $\mathfrak{P}$ and a generalized metric $G$ lies in the fact that the eigenspace corresponding to $+1$, $V^{+}$, for $\mathfrak{P}$, is no longer necessarily positive definite. In general, however, we can study the existence conditions for a generalized almost para-complex structure and construct its induced connections, analogously to what is done for a generalized metric.
\\ A generalized almost product structure can be viewed as an endomorphism of the generalized tangent bundle:
$$\mathfrak{P}=\left[\begin{matrix}
H && \alpha \\ 
\beta && K \end{matrix}\right]$$
such that it satisfies the following condition:
\begin{gather}
\begin{cases}
\beta \alpha =I-K^2 \\ 
\alpha \beta =I-H^2 \\ 
H\alpha =-\alpha K \\ 
\beta H=-K\beta \\ 
\beta =-\beta ^{*} \\ 
\alpha =-\alpha ^{*}
\end{cases}.
\end{gather}
 Moreover we denote by $V^{\pm}$ the eigenbundles of $\mathfrak{P}$ and we suppose they have the same rank (or $\mathfrak{P}$ is a generalized almost para-complex structure). Observe that if $\alpha $ is invertible and $v=X+\xi \in V^{+}$ then it can be written as $v=Y+\alpha ^{-1}X-\alpha ^{-1}H(Y)$ where $Y=H(X)+\alpha (\xi )$. This provides a way of expressing the elements of $V^{+}$. From now on, we will assume that $\alpha$ is \emph{invertible}. 
\\ If we wish to define an element of $V^{+}$ as an extension of a vector field $X$, we give the following definition:
$$X^{+}=X+\alpha ^{-1}(X)-\alpha ^{-1}H(X).$$
Similarly, if we wish to extend to $V^{-}$:
$$X^{-}=X-\alpha ^{-1}(X)-\alpha ^{-1}H(X).$$
So we can now define the affine connection $\nabla ^{+}$ on $V^{+}$. Let $ w=Y+\eta \in V^{+}$, we define $\nabla ^{+}$ as:
$$\nabla _{v}^{+}w=\nabla _{X+\xi}^{+}(Y+\eta) =([X^{-},w]_{C})_{+},$$
where the subscript $+$ denotes the projection from $E$ onto $V^{+}$.
Via $\pi $, the connection on $V^{+}$ in turn defines an affine connection $\nabla ^{+}$ on $TM$: $$\nabla _{X}^{+}Y=\pi (\nabla _{X}^{+}Y^{+}) .$$
Thus, it is also possible to define the bracket on the generalized tangent bundle induced by $\nabla ^{+}$:
$$[X+\xi ,Y+\eta ]_{\nabla ^{+}}=[X,Y]+\nabla _{X}^{+}\eta -\nabla _{Y}^{+}\xi $$
and the generalized connection $\tilde{\nabla}$:
$$\tilde{\nabla}_{X+\xi}Y+\eta =\nabla_{X}^{+}Y+\nabla_{X}^{+}\eta .$$
\end{oss}
\begin{lemma}
Let $\mathfrak{P}$ be a generalized almost para-complex structure as in $(1)$. Let $\nabla ^{+}$ be the induced connection on $TM$ and $\tilde{\nabla}$ the associated generalized connection. If $\nabla ^{+}\alpha =\nabla ^{+}H=0$ then $\tilde{\nabla}\mathfrak{P}=0$.
\end{lemma}
\begin{proof}
From the equalities in $(1)$, it follows that if $\nabla ^{+}\alpha =\nabla ^{+}H=0$. Then, since $\alpha$ is invertible, $\nabla^{+}K$ and $\nabla ^{+}\beta $ vanish as well, and consequently $\tilde{\nabla}\mathfrak{P}$ does too.
\end{proof}
\begin{oss}
This confirms what has been established for generalized metrics. Indeed, in that case $\alpha ^{-1}=g$ and $H=g^{-1}b$.
\end{oss}
\begin{proposition}
Let $\mathfrak{P}$ be a generalized almost para-complex structure and $\nabla ^{+}$ the induced connection on $TM$. If $\nabla ^{+}\alpha =\nabla ^{+}H=0$ and $d(\alpha ^{-1}H)=0$, then $\mathfrak{P}$ is $[\cdot,\cdot]_{\nabla ^{+}}$-integrable.
\end{proposition}
\begin{proof}
From Remark 5.2.1, we can state that, in general, if $\nabla ^{+}$ is torsion-free, a generalized almost para-complex structure parallel with respect to the generalized connection $\tilde{\nabla}$ induced by $\nabla ^{+}$, is also $[\cdot,\cdot]_{\nabla ^{+}}$-integrable. Indeed, for the Nijenhuis tensor of a generalized almost para-complex structure, the following equality holds:
$$N_{\mathfrak{P}}^{[\cdot,\cdot]_{\nabla ^{+}}}(X+\xi ,Y+\eta )= -T^{\nabla ^{+}}(X,Y)+\mathfrak{P}(T^{\nabla ^{+}}(\pi_{TM}\mathfrak{P}X,Y))+$$
$$ +\mathfrak{P}(T^{\nabla ^{+}}(X,\pi_{TM}\mathfrak{P}Y))
 -T^{\nabla ^{+}}(\pi_{TM}\mathfrak{P}X,\pi_{TM}\mathfrak{P}Y).$$
We now compute the torsion of $\nabla ^{+}$, $T^{\nabla ^{+}}(X,Y)$. Observe that:
$$\nabla _{X}^{+}Y=\pi (\nabla _{X}^{+}Y^{+})=\pi (([X-\alpha ^{-1}(X)-\alpha ^{-1}H(X),Y+\alpha ^{-1}(Y)-\alpha ^{-1}H(Y)]_{C})_{+}).$$
Knowing that the projection onto $V^{+}$, is $\frac{1}{2}(Id+\mathfrak{P})$, we obtain:
$$\nabla _{X}^{+}Y=\frac{1}{2}([X,Y]+H([X,Y])+\alpha (\mathcal{L}_{X}(\alpha ^{-1}(Y)-\alpha ^{-1}H(Y))-\mathcal{L}_{Y}(-\alpha ^{-1}(X)+$$ 
$$-\alpha ^{-1}H(X))-\frac{1}{2}d(\iota_{X}(\alpha ^{-1}(Y)-\alpha ^{-1}H(Y))-\iota _{Y}(-\alpha ^{-1}(X)-\alpha ^{-1}H(X))))).$$
Similarly, one can determine:
$$\nabla _{Y}^{+}X=\frac{1}{2}([Y,X]+H([Y,X])+\alpha (\mathcal{L}_{Y}(\alpha ^{-1}(X)-\alpha ^{-1}H(X))-\mathcal{L}_{X}(-\alpha ^{-1}(Y)+$$ 
$$-\alpha ^{-1}H(Y))-\frac{1}{2}d(\iota_{Y}(\alpha ^{-1}(X)-\alpha ^{-1}H(X))-\iota _{X}(-\alpha ^{-1}(Y)-\alpha ^{-1}H(Y))))).$$
Hence, the torsion is:
$$T^{\nabla ^{+}}(X,Y)=\nabla _{X}^{+}Y-\nabla _{Y}^{+}X-[X,Y]=$$ $$=H([X,Y])+\alpha (\mathcal{L}_{X}(-\alpha ^{-1}H(Y))-\mathcal{L}_{Y}(-\alpha ^{-1}H(X))+d(\iota _{X}\iota _{Y}(\alpha ^{-1}H))) =$$
$$=\alpha (\iota_{X}\iota_{Y}d(\alpha ^{-1}H)),$$
where in the last equality we used the following properties of the Lie derivative:
$$\mathcal{L}_{X}\omega =d(\iota _{X}\omega )+\iota _{X}d\omega ,$$
$$\mathcal{L}_{X}(\iota _{Y}\omega )=\iota _{Y}(\mathcal{L}_{X}\omega )+\iota _{[X,Y]}\omega .$$
Thus, if $d(\alpha ^{-1}H)=0$, it follows that $T^{\nabla ^{+}}=0$ and consequently $T^{\tilde{\nabla}}=0$, from which we obtain the $[\cdot ,\cdot ]_{\nabla ^{+}}$-integrability of $\mathfrak{P}$.
\end{proof}
\begin{oss}
In this case as well, we find analogies with generalized metrics. Indeed, in the case of a generalized metric, we have $\alpha ^{-1}H=-b$.
\end{oss}
\begin{proposition}
Let $\mathfrak{P}$ be a generalized almost para-complex structure as in $(1)$. Let $\nabla ^{+}$ be the induced connection on $TM$ and $\tilde{\nabla}$ the associated generalized connection. If $\nabla ^{+}\alpha =\nabla ^{+}H=0$, $d(\alpha ^{-1}H)=0$ and $$\langle \tilde{\nabla}_{u}v,w\rangle =\langle \tilde{\nabla}_{\mathfrak{P}u}v,\mathfrak{P}w\rangle \hspace{0.2cm}\forall v,w \in C^{\infty}(E),$$
then $\mathfrak{P}$ is $[\cdot,\cdot]_{C}$-integrable.
\end{proposition}
\begin{proof}
The proof follows the same steps as in Remark 5.1.1 and then applies Lemma 6.2.1 and Proposition 6.1.2.
\end{proof}
\begin{proposition}
Let $G$ be a generalized metric induced by $(g,b)$ where $g$ is a pseudo-Riemannian metric and $b$ a 2-form. Then, the following conditions hold:
\begin{nicenum}[label=\roman*.]
\item If $\nabla ^{+}b=0$ and $db=0$ then $G$ is $[\cdot,\cdot]_{\nabla ^{+}}$-integrable.
\item If $\nabla ^{+}b=0$, $db=0$ and 
$$\langle \tilde{\nabla}_{u}v,w\rangle =\langle \tilde{\nabla}_{\mathfrak{P}u}v,\mathfrak{P}w\rangle \hspace{0.2cm}\forall v,w \in C^{\infty}(E)$$
then $G$ is $[\cdot,\cdot]_{C}$-integrable.
\end{nicenum}
\end{proposition}
\begin{proof}
We apply Propositions 6.1.2 and 6.2.2, observing that in this case $\alpha ^{-1}H=b$, $\alpha ^{-1}=g$ and $\nabla ^{+}g=0$ by construction of $\nabla ^{+}$ and since $g$ is non-degenerate.
\end{proof}
\begin{oss}
Let $g$ be a pseudo-Riemannian metric and let $sign(g)=(p,q)$ be its signature, where $0<p\leq q <m$, $p+q=dim_{\mathbb{R}}M=m$. Therefore, an isotropic subbundle $D\subset TM$ may exist.
\\ In particular, one can obtain an isotropic subbundle of dimension equal to $p$.
\end{oss}
%\begin{lemma}
%Let $g$ be a pseudo-Riemannian metric of $sign(g)=(p,q)$ with $0<p\leq q <m$, $p+q=dim_{\mathbb{R}}M=m$. 
%The following conditions are sufficient for the existence of an isotropic subbundle $D\subset TM$ of dimension $k\leq p$:
%\begin{nicenum}
%\item $TM =\epsilon ^{k}\oplus E'$ where $\epsilon ^{k}$ is a trivial subbundle of rank $k$.
%\item There exist $k$ vector fields $X_{1}, \dots ,X_{k} \in C^{\infty}(TM)$ such that \\ $g(X_{i},X_{j})=0 \hspace{0.2cm} \forall i,j \in \{1, \dots , k\}$.
%\item There exist $k$ linear indipendent sections of $TM$.
%\end{nicenum}
%\end{lemma}
\begin{lemma}
Let $g$ be a pseudo-Riemannian metric of $sign(g)=(p,q)$ with $0<p\leq q <m$, $p+q=dim_{\mathbb{R}}M=m$. If there exist $k$ smooth vector fields $X_{1},\dots ,X_{k} \in C^{\infty}(TM)$ which are pointwise linearly independent and satisfy $$g(X_{i},X_{j})=0 \hspace{0.25cm} \forall i,j \in \{1, \dots k\},$$
then $D:=span\{X_{1},\dots ,X_{k}\}$ is a rank-$k$ subbundle of $TM$ which is totally isotropic; in particular such a $k$ must satisfy $k\leq p$.
\end{lemma}
\begin{proof}
Pointwise independence of the $X_{i}$ guarantees that their span defines a smooth rank $k$ subbundle $D$. The mutual nullity $g(X_{i},X_{j})=0$ implies every fiber $D_{x}$ is isotropic. Hence $D$ is an isotropic subbundle and since the maximal dimension of an isotropic subbundle in signature $(p,q)$ is $p$, we have $k\leq p$.
\end{proof}
\begin{lemma}
Let $g$ be a pseudo-Riemannian metric of $sign(g)=(p,q)$ with $0<p\leq q <m$, $p+q=dim_{\mathbb{R}}M=m$. Locally, there always exists an isotropic subbundle $D_{max}$ of maximal dimension $p$.
\end{lemma}
\begin{proof}
At $x_{0}$, choose a maximal totally isotropic subbundle $W \subset T_{x_{0}}M$ of dimension $p$ and let $w_{1}, \dots ,w_{p}$ be a basis of $W$. Take a normal geodesic neighbourhood $U$ of $x_{0}$. Parallel transport each $w_{i}$ along the unique radial geodesic from $x_{0}$ to $x \in U$; the resulting fields $W_{i}$ are smooth, remain pointwise linearly independent and, since parallel transport preserves $g$, satisfy $g(W_{i},W_{j})=0$ $\forall i,j$ on $U$. Hence the span of $W_{i}$ defines a smooth rank-$p$ subbundle $D_{max} \subset TM|_{U}$ which is isotropic.
\end{proof}
\begin{proposition}
Let $g$ be a pseudo-Riemannian metric of $sign(g)=(p,q)$ with $0<p\leq q <m$, $p+q=dim_{\mathbb{R}}M=m$. If there exists a global isotropic subbundle of dimension $p$, then a maximal isotropic subbundle in $E$ can be constructed starting from $D_{max}$. Locally, this construction always exists.
\end{proposition}
\begin{proof}
Given an isotropic subbundle $D_{max}$, it is possible to define $v(D_{max})$, a subbundle of $V^{+}$, as the extension of the elements of $D_{max}$ to $V^{+}$. That is, if $X \in C^{\infty}(D_{max})$, we define $X^{+}=X+gX+bX \in v(D_{max})$. Moreover, we observe that if $X \in C^{\infty}(D_{max})$, then $g(X,X)=0$ and hence:
$$\langle X^{+},X^{+}\rangle =g(X,X)=0$$
and if $X,Y \in C^{\infty}(D_{max})$ we have:
$$\langle X^{+},Y^{+}\rangle =g(X,Y)=0.$$
Therefore, the subbundle $v(D_{max}) \subset V^{+}$ is $\langle \cdot ,\cdot \rangle $-isotropic.
\\ We observe, however, that neither $D_{max}$ or $v(D_{max})$ are a priori maximal isotropic subbundles. If $g$ were a neutral metric, then one could conclude that $v(D_{max})$ is a maximal isotropic subbundle of $V^{+}$.
\\ Let us now consider $Ann(D_{max})=\{\alpha \in C^{\infty}(T^{*}M): \alpha |_{D_{max}}=0\} \subset Ann(v(D_{max}))$ and observe that if $\alpha ,\beta \in Ann(D_{max})$ then $\langle \alpha ,\beta \rangle =0$.
Hence, the annihilator $Ann(D_{max})$ is $\langle \cdot ,\cdot \rangle $-isotropic.
\\ We thus define $V_{i}=v(D_{max})\oplus Ann(D_{max})$ and consider $X^{+}\in V(D_{max})$, $\alpha \in Ann(D_{max})$:
$$\langle X^{+},\alpha \rangle =\frac{1}{2}(\alpha (X))=0.$$
Hence, $V_{i}$ is an isotropic subbundle of $E$ of dimension $dimV_{i}=p+(m-p)=m=dim_{\mathbb{R}}M$, and therefore maximal.
\end{proof}
\begin{proposition}
Let $G$ be a generalized metric induced by $(g,b)$ where $g$ is a pseudo-Riemannian metric of $sign(g)=(p,q)$ with $0<p\leq q <m$, $p+q=dim_{\mathbb{R}}M=m$. 
\\ Suppose there exists a global isotropic subbundle $D_{max}\subset TM$ of dimension $p$. If $G$ is $[\cdot,\cdot]_{C}$-integrable and $D_{max}$ is $[\cdot,\cdot]$-involutive, then the strong generalized almost complex structure $\mathfrak{J}_{G}$ induced by $G$, that is, the one associated with $V_{i}=v(D_{max})\oplus Ann(D_{max})$, is $[\cdot,\cdot]_{C}$-integrable.
\end{proposition}
\begin{proof}
To verify the integrability of the strong generalized almost complex structure, we show that $V_{i}$ is $[\cdot,\cdot]_{C}$-involutive.
\\ Let $\alpha , \beta $ be in $Ann(D_{max})$, then $[\alpha ,\beta ]_{C}=0$.
\\ Let us now consider $X^{+} \in v(D_{max})$ and $\alpha$ in $Ann(D_{max})$: $[X^{+},\alpha ]_{C}=\mathcal{L}_{X}(\alpha )$.
\\ We now observe that $$\mathcal{L}_{X}(\alpha ) \in V_{i} \iff \mathcal{L}_{X}(\alpha ) \in Ann(D_{max}) \iff \mathcal{L}_{X}(\alpha )(Y)=0 \hspace{0.2cm} \forall Y\in C^{\infty}(D_{max}).$$
Carrying out the calculations, we obtain:
$$(\mathcal{L}_{X}\alpha )(Y)=X(\alpha (Y))-\alpha ([X,Y])=-\alpha ([X,Y])$$
but since $D_{max}$ is $[\cdot,\cdot]$-involutive, $[X,Y]\in D_{max}$ and so $\mathcal{L}_{X}\alpha $ vanishes, or \\ $\mathcal{L}_{X}\alpha \in V_{i} \hspace{0.35cm} \forall X \in C^{\infty}(D_{max}), \hspace{0.2cm} \forall \alpha \in C^{\infty}(Ann(D_{\max})).$
Finally since $G$ is integrable, it follows that $V^{+}$ is $[\cdot,\cdot]_{C}$-involutive, and therefore $v(D_{max})$ is $[\cdot,\cdot]_{C}$-involutive. Hence $V_{i}$ is $[\cdot,\cdot]_{C}$-involutive and $\mathfrak{J}_{G}$ is integrable.
\end{proof}

\section{Weak metric structures}
\subsection{Weak almost Hermitian manifolds}
Following \cite{12} we pose:
\begin{definition}
A \emph{weak almost Hermitian manifold} $(M,g,A)$ is a Riemannian manifold $(M,g)$ of dimension $m=2n \geq 4$ endowed with a $g$-skew-symmetric non-singular endomorphism $A:TM\rightarrow TM$. $F(X,Y):=g(AX,Y)$ is called $2$-fondamental form of $(M,g,A)$.
\end{definition}
\begin{oss}
Let $(M,g,A)$ be a weak almost Hermitian manifold. Then $Q=-A^{2}$ is $g$-self-adjoint.
\end{oss}
\begin{proposition}
Let $M$ be a smooth manifold of dimension $m\geq 4$ and let $E=TM\oplus T^{*}M$. Let $G$ be a generalized metric on $E$ and let $(g,b)$ be the induced structures. If $b$ is non-degenerate then $m$ is even and $G$ induces a weak almost Hermitian structures on $M$. Conversely, if $(M,g,A)$ is a weak almost Hermitian structure then $g$ and $b(X,Y):=g(AX,Y)$, $X,Y \in C^{\infty}(TM)$ define a generalized metric on $E$.
\end{proposition}
\begin{proof}
Given a generalized metric $G$, one can define the following $(1,1)$-tensor $A$:
$$g(AX,Y)=b(X,Y) \hspace{0.2cm} \forall X,Y \in C^{\infty}(TM).$$
$A$ is an endomorphism of $TM$, moreover, since $g$ and $b$ are non-degenerate, $A$ is also non-degenerate, and it can be shown that it is $g$-skew-symmetric as well:
$$g(AX,Y)=b(X,Y)=-b(Y,X)=-g(AY,X)=-g(X,AY) \hspace{0.2cm} \forall X,Y \in C^{\infty}(TM).$$
Let us now consider $Q=-A^{2}$: it is clearly non-degenerate, and $g$-self-adjoint, moreover:
$$g(AX,AY)=g(-A^{2}X,Y)=g(QX,Y) \hspace{0.2cm} \forall X,Y \in C^{\infty}(TM).$$
Conversely $b(X,Y)=g(AX,Y)=-g(X,AY)=b(Y,X)$ is a $2$-form.
\end{proof}
\begin{definition}
$\cite{12}$ Let $G$ be a generalized metric on $M$ and let $A$ be a skew-symmetric endomorphism of $TM$ of constant rank. A connection $\nabla$ on $M$ is said to have \emph{the $A$-torsion condition} if its torsion tensor $T^{\nabla}$ satisfies:
$$T^{\nabla}(AX,Y)=T^{\nabla}(X,AY) \hspace{0.2cm} \forall X,Y \in C^{\infty}(TM).$$
\end{definition}
\begin{oss}
We have already seen in previous proposition that, given a generalized metric $G$, it is possible to find a skew-symmetric endomorphism $A$ associated with $G$. In particular if $b$ is non-degenerate, then $A$ also has constant rank.
\end{oss}
\begin{lemma}
Let $G$ be a generalized metric on $M$. Let $(g,b,\nabla ^{+})$ be the induced structures and $A$ the induced endomorphism. If $\nabla ^{+}b=0$ and $b$ satisfies the following equality:
$$b(g^{-1}(db(X,AY))-g^{-1}(db(AX,Y)),Z)+b(g^{-1}(db(AY,Z)),X)+$$
$$-b(g^{-1}(db(Y,Z)),AX)-b(g^{-1}(db(Z,AX)),Y)+b(g^{-1}(db(Z,X)),AY)=0,$$
then $\nabla ^{+}$ have the $A$-torsion condition.
\end{lemma}
\begin{proof}
The $A$-torsion condition is equivalent to the following equality:
$$0=T^{\nabla^{+}}(AX,Y,Z)-T^{\nabla ^{+}}(X,AY,Z)=-db(AX,Y,Z)+db(X,AY,Z)$$ $$ \forall X,Y,Z \in C^{\infty}(TM).$$
We remark that, in general, the following holds:
$$db(X,Y,Z)=\sum _{cyc}((\nabla _{X}^{+}b)(Y,Z)-b(T^{\nabla ^{+}}(X,Y),Z)).$$
Thus, we can substitute this equality into $db(AX,Y,Z)$ and $db(X,AY,Z)$, remarking that $\nabla ^{+}b=0$ and $T^{\nabla ^{+}}(X,Y)=-g^{-1}(db(X,Y))$, we have the statement.
\end{proof}
\begin{definition}
$\cite{12}$ Let $G$ be a generalized metric on $M$, $\nabla$ an affine connection on $M$ and $Q:TM\rightarrow TM$ a $g$-self-adjoint endomorphism. $\nabla$ is said to have \emph{the $Q$-torsion condition} if its torsion tensor $T^{\nabla}$ satisfies:
$$T^{\nabla}(QX,Y)=T^{\nabla}(X,QY)=Q(T^{\nabla}(X,Y)) \hspace{0.35cm} \forall X,Y \in C^{\infty}(TM).$$
\end{definition}
\begin{lemma} Let $G$ be a generalized metric and $(g,b,\nabla ^{+})$ the induced structures. Let $A,Q$ be the endomorphism be defined as in the proof of Proposition 7.1.1. If $\nabla ^{+}$ have the $A$-torsion condition and $db(QX,Y,Z)=db(X,Y,QZ) \hspace{0.2cm} \forall X,Y \in C^{\infty}(TM)$, then $\nabla ^{+}$ have also the $Q$-torsion condition.
\end{lemma}
\begin{proof}
Recall that $Q=-A^{2}$. Since $\nabla ^{+}$ have the $A$-torsion condition, we have that $$T^{\nabla ^{+}}(QX,Y)=T^{\nabla ^{+}}(X,QY).$$
Moreover we know that $$g(T^{\nabla ^{+}}(QX,Y),Z)=T^{\nabla ^{+}}(QX,Y,Z)=-db(QX,Y,Z),$$
$$g(QT^{\nabla ^{+}}(X,Y),Z)=T^{\nabla ^{+}}(X,Y,QZ)=-db(X,Y,QZ).$$
Thus, from the assumption on $db$, it follows that $T^{\nabla ^{+}}(QX,Y)=QT^{\nabla ^{+}}(X,Y) \hspace{0.2cm}$ \\ $\forall X,Y \in C^{\infty}(TM)$.
\end{proof}

\subsection{Weak nearly K\"ahler manifolds}
\begin{definition}
$\cite{12}$ A weak almost Hermitian manifold is said to be \emph{weak nearly K\"ahler} if:
$$(\nabla _{X}^{LC}A)X=0 \iff (\nabla_{X}^{LC}F)(X,Y)=0 \hspace{0.2cm} \forall X,Y \in C^{\infty}(TM).$$
If $\nabla ^{LC}A=0$, then such structure is called \emph{weak K\"ahler manifold}.
\end{definition}
\begin{proposition}
$\cite{12}$ Let $(M,g,A,Q)$ be a weak almost Hermitian manifold. Let $b(X,Y)=g(AX,Y)$ and let $G$ be the generalized metric induced by $(g,b)$. Let $\nabla$ be an affine connection with totally skew-symmetric torsion and such that $\tilde{\nabla}G=0$, where $\tilde{\nabla}$ is the generalized connection induced by $\nabla$. Then $\nabla$ has the $A$-torsion condition if and only if $(M,g,A,Q)$ is a weak nearly K\"ahler manifold.
\end{proposition}
\begin{proposition}
Let $G$ be a generalized metric on $M$, $dimM=2n\geq 4$, and $(g,b,A,Q)$ the induced structures as in Proposition 7.7.1.
If $b$ is non-degenerate, $\nabla ^{+}b=0$ and $b$ satisfies the following equality:
$$b(g^{-1}(db(X,AY))-g^{-1}(db(AX,Y)),Z)+b(g^{-1}(db(AY,Z)),X)+$$
$$-b(g^{-1}(db(Y,Z)),AX)-b(g^{-1}(db(Z,AX)),Y)+b(g^{-1}(db(Z,X)),AY)=0,$$
then $G$ induces a weak nearly K\"ahler manifold $(M,g,A,Q)$.
\end{proposition}
\begin{proof}
From Proposition 7.7.1 and Lemma 7.2.1, we obtain that $G$ induces a weak almost Hermitian structure and that $\nabla ^{+}$ satisfies the $A$-torsion condition. Moreover $\nabla ^{+}g=\nabla ^{+}b=0$; hence, by Proposition 7.4.1, the claim follows.
\end{proof}

\section{Generalized K\"ahler structures}
In this section, we will consider only strong generalized almost complex / complex structures, and therefore we will omit specifying “strong".
\begin{definition} 
\cite{1} Let $\mathfrak{J}$ be a generalized almost complex structure and let $G$ be a generalized metric. The metric $G$ is a \emph{generalized Hermitian metric} with respect to $\mathfrak{J}$ if 
$$\langle G\mathfrak{J}v,\mathfrak{J}w\rangle =\langle Gv,w\rangle \ \forall v,w \in C^{\infty}(E).$$
In this case we call the pair $(\mathfrak{J},G)$ a \emph{generalized Hermitian structure}.
\end{definition}
\begin{oss}
This means that the generalized metric $G$ must be compatible with $\mathfrak{J}$. Equivalently, $G$ is a generalized Hermitian metric if and only if it commutes with $\mathfrak{J}$, i.e. $G\mathfrak{J}=\mathfrak{J}G$. Hence the eigenbundles of $G$, \\ $V^{\pm}=\{X\pm g(X)+b(X)|X\in C^{\infty}(TM)\}$ are invariant under $\mathfrak{J}$.
\\ Moreover, from the commutativity of $G$ and $\mathfrak{J}$ it follows that $G\mathfrak{J}$ is a generalized almost complex structure that anticommutes with $\mathfrak{J}$.
\end{oss}
\begin{definition} 
\cite{7,8} A \emph{generalized K\"ahler structure} is a pair of generalized complex structure $\mathfrak{J}_{1}$, $ \mathfrak {J}_{2}$ which commute and such that $-\mathfrak{J}_{1}\mathfrak{J}_{2}$ is a generalized metric.
\end{definition}
\begin{oss}
\cite{7} Equivalently, a generalized K\"ahler structure may be described as a generalized complex structure $\mathfrak{J}$ together with a generalized Hermitian metric $G$ such that $G\mathfrak{J}$ is again a generalized complex structure.
\end{oss}
\begin{oss}
\cite{7} Let $G$ be a generalized metric and let $V^{\pm}$ denote its eigenbundles. As already observed, there are bundle isomorphism $\pi :V^{\pm}\rightarrow TM$. Given a generalized almost complex structure $\mathfrak{J}$, it induces two almost complex structures $J_{\pm}:TM\rightarrow TM$ by:
$$(J_{\pm}X)^{\pm}=\mathfrak{J}X^{\pm}.$$
or equivalently $J_{\pm}X=\pi (\mathfrak{J}X^{\pm})$, where the superscript $\pm$ denotes the extension of a vector field to $V^{\pm}$.
\\ Thus, if $(\mathfrak{J}_{1},\mathfrak{J}_{2})$ is a generalized K\"ahler structure, then $\mathfrak{J}_{1}$ induces two Hermitian structure $J_{\pm}$ on $TM$, while $\mathfrak{J}_{2}$ induces $J_{+}$ and $-J_{-}$.
\end{oss} 
\hspace{-0.62cm} We will study $[\cdot,\cdot]_{C}$-integrability of a generalized Hermitian structure. This is equivalent to find sufficient conditions under which a generalized Hermitian structure is a generalized K\"ahler structure. Later we will also analyze $[\cdot,\cdot]_{\nabla ^{+}}$-integrability for these structures.
\begin{proposition} Let $(G,\mathfrak{J})$ be a generalized Hermitian structure. Let $g$ be the Riemannian metric induced by $G$ and let $b$ be the $2$-form induced by $G$. Let $(M,g,J_{\pm})$ be the almost Hermitian structures on $TM$ induced by $\mathfrak{J}$. If $db=0$ and $\nabla ^{+}J_{+}=\nabla ^{+}J_{-}=0$, then $(G,\mathfrak{J})$ is a generalized K\"ahler structure.
\end{proposition}
\begin{proof}
Since $db=0$, by Proposition 5.2.1 we have $\nabla ^{+}=\nabla ^{-}=\nabla ^{LC}$. From the hypothesis on $J_{\pm}$ it follows that $(M,g,J_{\pm})$ are K\"ahler manifolds. Consider the generalized almost complex structures:
$$\mathfrak{J}_{1}=\frac{1}{2}e^{b}\left( \begin{matrix}
J_{+}+J_{-} && -(\omega _{+}^{-1}-\omega _{-}^{-1}) \\ 
\omega _{+}-\omega _{-} && -(J_{+}^{*}+J_{-}^{*})
\end{matrix} \right) e^{-b},$$
$$\mathfrak{J}_{2}=\frac{1}{2}e^{b}\left( \begin{matrix}
J_{+}-J_{-} && -(\omega _{+}^{-1}+\omega _{-}^{-1}) \\ 
\omega _{+}+\omega _{-} && -(J_{+}^{*}-J_{-}^{*}) 
\end{matrix} \right) e^{-b},$$
where $\omega _{\pm}$ are the K\"ahler forms associated to $(g,J_{\pm})$. These structures are $[\cdot,\cdot]_{C}$-integrable since they arise from K\"ahler manifolds. Hence $(\mathfrak{J}_{1},\mathfrak{J}_{2})$ is a generalized K\"ahler structure. Now we observe that $\mathfrak{J}$ and $\mathfrak{J}_{1}$ induce the same complex structures on $TM$, so:
$$\pi (\mathfrak{J}X^{+})=J_{+}X=\pi (\mathfrak{J}_{1}X^{+}),$$
$$\pi (\mathfrak{J}X^{-})=J_{-}X=\pi (\mathfrak{J}_{1}X^{-}).$$
Since $\pi$ is an isomorphism, $\mathfrak{J}$ and $\mathfrak{J}_{1}$ coincide up to a $b$-field transform on the whole $E$, and therefore $\mathfrak{J}$ is $[\cdot,\cdot]_{C}$-integrable. By the same argument applied to $\mathfrak{J}_{2}$ one obtains the $[\cdot,\cdot]_{C}$-integrability of $G\mathfrak{J}$, hence $(G,\mathfrak{J})$ is a generalized K\"ahler structure.
\end{proof}
\begin{lemma}
\cite{1,7} Let $(\mathfrak{J},G)$ be a generalized K\"ahler structure. Then we can write
 $$E\otimes \mathbb{C}=V_{i}^+ \oplus V_{-i}^+ \oplus V_{i}^- \oplus V_{-i}^-$$ 
in terms of the simultaneous eigenbundles of $\mathfrak{J}$ and $G$. Each eigenbundle is involutive and isotropic.
\end{lemma}
\begin{proposition} 
Let $(G,\mathfrak{J})$ be a generalized K\"ahler structure. Then $\nabla^{\pm}\mathfrak{J}$ vanishes, hence $J_{\pm}$ are integrable and $\nabla^{+}J_{+}=\nabla ^{-}J_{-}=0$.
\end{proposition}
\begin{proof} 
We first prove the integrability of $J_{\pm}$ using integrability of $\mathfrak{J}$. Denote by $V_{i}$ and  $V_{-i}$ the eigenbundles of $\mathfrak{J}$. Let $X,Y$ be vector fields with $J_{+}X=iX$ and $J_{+}Y=iY$. Since $V_{i}$ is $[\cdot,\cdot]_{C}$-involutive, we have
$$\mathfrak{J}([X^{+},Y^{+}]_{C})=i[X^{+},Y^{+}]_{C}.$$ 
Noting that $\pi ([X^{+},Y^{+}]_{C})=[X,Y]$ (i.e. $[X^{+},Y^{+}]=[X,Y]^{+}$), it follows that:
$$(J_{+}[X,Y])^{+}=i[X,Y]^{+}$$
hence $J_{+}[X,Y]=i[X,Y]$, so $J_{+}$ is integrable. The same reasoning applies to $J_{-}$.
\\ Next we show $\nabla ^{\pm}\mathfrak{J}=0$. Let $v=Y^{+} \in V^{+}$ and compute
$$\mathfrak{J}(\nabla_{X}^{+}v)=\mathfrak{J}([X^{-},v]_{+})=\mathfrak{J}([X^{-},Y^{+}]_{+}).$$
If $X^{-},Y^{+}$ are sections of the same eigenbundle of $\mathfrak{J}$, then, since these eigenbundles are involutive, we have
$$\mathfrak{J}(\nabla_{X}^{+}v)=\nabla_{X}^{+}(\mathfrak{J}v).$$
Now consider $X^{-}\in C^{\infty}(V_{-i})$ and $Y^{+}\in C^{\infty}(V_{i})$. Because $(G,\mathfrak{J})$ is generalized K\"ahler, $G\mathfrak{J}$ is integrable, and so:
$$N_{G\mathfrak{J}}(X^{-},Y^{+})=2([iX^{-},iY^{+}]_{C}+\mathfrak{J}([-iX^{-},Y^{+}]_{C})_{+}+\mathfrak{J}([X^{-},iY^{+}]_{C})_{-})=0.$$ 
From this and the fact that $\mathfrak{J}(V^{\pm})\subset V^{\pm}$, we obtain:
$$\mathfrak{J}([X^{-},Y^{+}]_{+})=i[X^{-},Y^{+}]_{+}, \hspace{0.4cm} \mathfrak{J}([X^{-},Y^{+}]_{-})=-i[X^{-},Y^{+}]_{-}.$$
and therefore again $\mathfrak{J}(\nabla _{X}^{+}v)=\nabla _{X}^{+}(\mathfrak{J}v)$. The analogous argument for the remaining cases yields $\nabla ^{+}\mathfrak{J}=0$; similarly one obtains $\nabla ^{-}\mathfrak{J}=0$.
\\ From these equalities we get:
$$(J_{+}\nabla_{X}^{+}Y)^{+}=\mathfrak{J}(\nabla_{X}^{+}Y)^{+}=\mathfrak{J}(\nabla_{X}^{+}Y^{+})=\nabla_{X}^{+}\mathfrak{J}(Y^{+})=\nabla_{X}^{+}(J_{+}Y)^{+}=(\nabla_{X}^{+}J_{+}Y)^{+}$$ 
hence $J_{+}\nabla_{X}^{+}Y=\nabla_{X}^{+}J_{+}(Y)$. The same equalities hold for $J_{-}$, and thus \\ $\nabla ^{-}J_{-}=0$.
\end{proof}
\begin{oss}
If $(G,\mathfrak{J})$ is a generalized K\"ahler structure and $(db)^{1,1}=0$, then $\nabla ^{\pm}$ are the Chern connections for $J_{\pm}$, respectively. Indeed, for a generalized K\"ahler structure $(G,\mathfrak{J})$ we have $\nabla ^{\pm}g=0$ and, as shown above, $\nabla ^{+}J_{+}=\nabla ^{-}J_{-}=0$. If additionally the $(1,1)$-component of $db$ vanishes, then the $(1,1)$-component of the torsions of $\nabla ^{\pm}$ vanishes as well, so $\nabla ^{\pm}$ coincide with the Chern connections of $J_{+}$ e $J_{-}$.
\end{oss}
\begin{proposition}
Let $(\mathfrak{J},G)$ be a generalized Hermitian structure. Let $g$ be the Riemannian metric and $b$ the $2$-form induced by $G$ and let $J_{\pm}$ be the almost complex structures on $TM$ induced by $\mathfrak{J}$. Let $\tilde{\nabla}$ be the generalized connection induced by $\nabla ^{+}$. If $\nabla  ^{+}b=0$ then:
$$\tilde{\nabla }\mathfrak{J}=0 \iff \nabla ^{+}J_{\pm}=0.$$
\end{proposition}
\begin{proof} 
Write the generalized almost complex structure in block form:
$$\mathfrak{J}=\left[\begin{matrix}
H && \alpha \\ 
\beta && K \\ \end{matrix}\right].$$
For $v=X+\xi $ and $w=Y^{+} \in V^{+}$ one computes:
$$\tilde{\nabla }_{v}\mathfrak{J}Y^{+}=\nabla _{X}^{+}((H+\alpha g+\alpha b)(Y))+\nabla _{X}^{+}((\beta +Kg+Kb)(Y)),$$
$$\mathfrak{J}\tilde{\nabla }_{v}Y^{+}=H(\nabla _{X}^{+}Y)+\alpha (\nabla _{X}^{+}(gY+bY))+\beta (\nabla _{X}^{+}Y)+K(\nabla _{X}^{+}(gY+bY)).$$
From the definition of the induced structures and by using $\nabla ^{+}b=0$, we have
$$\tilde{\nabla }_{v}Y^{+}=\nabla _{X}^{+}Y+\nabla _{X}^{+}(gY+bY)=\nabla _{X}^{+}Y+(g+b)(\nabla _{X}^{+}Y)=(\nabla _{X}^{+}Y)^{+},$$
and therefore
$$\pi (\mathfrak{J}\tilde{\nabla }_{v}Y^{+})=\pi (\mathfrak{J}((\nabla_{X}^{+}Y)^{+}))=J_{+}(\nabla _{X}^{+}Y),$$
$$\pi (\tilde {\nabla }_{v}\mathfrak{J}Y^{+})=\nabla _{X}^{+}(H(Y)+\alpha (g(Y)+b(Y))=\nabla _{X}^{+}\pi (\mathfrak{J}Y^{+}))=\nabla _{X}^{+}(J_{+}Y).$$
Analogous equalities hold for $J_{-}$.
%$$\pi (\mathfrak{J}\tilde{\nabla }_{v}Y^{-})=\pi (\mathfrak{J}((\nabla _{X}^{+}Y)^{-}))=J_{-}(\nabla _{X}^{+}Y);$$
%$$ \pi (\tilde{\nabla}_{v}\mathfrak{J}Y^{-})=\nabla _{X}^{+}(J_{-}Y).$$
\\ $"\rightarrow "$ If $\mathfrak{J}$ is parallel with respect to $\tilde{\nabla}$, previous identities give:
$$J_{+}(\nabla _{X}^{+}Y)=\pi (\mathfrak{J}\tilde{\nabla }_{v}Y^{+})=\pi (\tilde {\nabla}_{v}\mathfrak{J}(Y^{+}))=\nabla _{X}^{+}(J_{+}Y),$$
$$J_{-}(\nabla _{X}^{+}Y)=\pi (\mathfrak{J}\tilde{\nabla }_{v}Y^{-})=$$
$$=\pi (\tilde {\nabla}_{v}\mathfrak{J}(Y^{-}))=\nabla _{X}^{+}(J_{-}Y),$$
hence $\nabla ^{+}J_{+}=\nabla ^{+}J_{-}=0$.
\\ $"\leftarrow "$ The converse follows by reversing the order of the equalities above: if $\nabla ^{+}J_{\pm}=0$ then $\tilde{\nabla}\mathfrak{J}=0$ on $V^{+}$ and $V^{-}$ and hence on all of $E$.
\end{proof}

\begin{proposition}
Let $(G,\mathfrak{J})$ be a generalized Hermitian structure. Let $g$ and $b$ be the metric and $2$-form induced by $G$ and let $J_{\pm}$ be the almost complex structures induced by $\mathfrak{J}$ on $TM$. If $\nabla^{+}b=0$, $\nabla ^{+}J_{\pm}=0$ and $db=0$, then both $G$ and $\mathfrak{J}$ are $[\cdot,\cdot]_{\nabla ^{+}}$-integrable.
\end{proposition}
\begin{proof}
The $[\cdot,\cdot]_{\nabla ^{+}}$-integrability of $G$ follows from Proposition 3.6.1. For $\mathfrak{J}$ we compute its Nijenhuis tensor with respect to $[\cdot,\cdot]_{\nabla ^{+}}$. 
\\ Let $v=X+\xi$, $w=Y+\eta \in C^{\infty}(E)$ and note that $\tilde{\nabla}\mathfrak{J}=0$, where $\tilde{\nabla}$ is the generalized connection induced by $\nabla ^{+}$. Then:
\\ $N_{\mathfrak{J}}^{\nabla ^{+}}(v,w)=\tilde{\nabla} _{\mathfrak{J}(v)}\mathfrak{J}(w)-\tilde{\nabla} _{\mathfrak{J}(w)}\mathfrak{J}(v) -T^{\tilde{\nabla}}_{[\cdot,\cdot]_{\nabla^{+}}}(\mathfrak{J}(v),\mathfrak{J}(w))-\mathfrak{J}(\tilde{\nabla}_{\mathfrak{J}(v)}w -\tilde{\nabla}_{w}\mathfrak{J}(v)+$
\\ $-T^{\tilde{\nabla}}_{[\cdot,\cdot]_{\nabla^{+}}}(\mathfrak{J}(v),w))-\mathfrak{J}(\tilde{\nabla}_{v}\mathfrak{J}(w) -\tilde{\nabla}_{\mathfrak{J}(w)}v-T^{\tilde{\nabla}}_{[\cdot,\cdot]_{\nabla^{+}}}(v,\mathfrak{J}(w))-\tilde {\nabla} _{v}w+\tilde{\nabla} _{w}v+T^{\tilde{\nabla}}_{[\cdot,\cdot]_{\nabla^{+}}}(v,w)=T^{\tilde{\nabla}}_{[\cdot,\cdot]_{\nabla^{+}}}(v,w)+\mathfrak{J}(T^{\tilde{\nabla}}_{[\cdot,\cdot]_{\nabla^{+}}}(v,\mathfrak{J}(w))+$
\\ $+\mathfrak{J}(T^{\tilde{\nabla}}_{[\cdot,\cdot]_{\nabla^{+}}}(\mathfrak{J}(v),w))-T^{\tilde{\nabla}}_{[\cdot,\cdot]_{\nabla^{+}}}(\mathfrak{J}(v),\mathfrak{J}(w))$.
\\ Recalling that $T_{[\cdot,\cdot]_{\nabla ^{+}}}^{\tilde{\nabla }}(X+\xi ,Y+\eta )=T^{\nabla ^{+}}(X,Y)$, it follows that if $db=0$, then the generalized connection $\tilde{\nabla }$ is symmetric with respect to $[\cdot,\cdot]_{\nabla ^{+}}$ and hence $\mathfrak{J}$ is $[\cdot,\cdot]_{\nabla ^{+}}$-integrable.
\end{proof}
\begin{corollary}
Let $(\mathfrak{J},G,G\mathfrak{J})$ be a generalized K\"ahler structure. If $\nabla ^{+}b=0$ and $db=0$, then $$\tilde{\nabla }G=\tilde{\nabla }\mathfrak{J}=\tilde{\nabla }G\mathfrak{J}=0,$$
and moreover $G$, $\mathfrak{J}$ and $G\mathfrak{J}$ are $[\cdot,\cdot]_{\nabla ^{+}}$-integrable.
\end{corollary}
\begin{proof} 
The claim follows from previous propositions: for a generalized K\"ahler structure with $\nabla ^{+}b=db=0$, we have $\nabla ^{+}J_{\pm}=0$, hence $\tilde{\nabla}G=\tilde{\nabla}\mathfrak{J}=0=\tilde{\nabla}G\mathfrak{J}$. A direct computation of the Nijenhuis tensor $N_{G\mathfrak{J}}^{[\cdot,\cdot]_{\nabla ^{+}}}$ shows that $G\mathfrak{J}$ is also $[\cdot,\cdot]_{\nabla ^{+}}$-integrable.
\end{proof}

\nocite{*}
\bibliographystyle{plain}
\bibliography{bibliografia}

@article{1,
author ={Baraglia, D.},
title = {Generalized Geometry},
journal = { The University of Adelaide},
year = {2007}
}

@article{2,
author ={Blaga, A. M. and Nannicini, A. }, 
title={Canonical connections attached to generalized quaternionic and para-quaternionic structures},
journal={ Revista de la Real Academica de Ciencias exactas, F\'isicas y Naturales. Serie A, Matem\'aticas, vol. 117, 4, article 150 },
year={2023},
note={ arXiv:2302.052391v1 [math.DG]}
}

@article{3,
author={Cortés, V. and David, L.}, 
title={Generalized connections, spinors, and integrability of generalized structures on Courant algebroids},
journal={Moscow Math J, 21 (4)}, 
year={2021},
pages ={695-736}, 
note={arXiv:1905.01977v4 [math.DG]}
}

@article{4,
author={Etayo, F. and Gomez-Nicolas, P. and Santamaria, R.},
title={About generalized complex structures on {$\mathbb{S}^{6}$}},
journal={ arXiv:2405.05681v2 [math.DG] (2024)}
}

@article{5, 
author={Etayo, F. and Gomez-Nicolas, P. and Santamaria, R.},
title={Metric polynomial structures on generalized geometry}, 
journal={Publ. Math. Debrecen},
year={2024}, 
pages={171-196}
}

@article{6,
author={Gualtieri, M.}, 
title={Branes on Poisson varieties},
journal={"The Many Facets of Geometry", Oxford University Press}, 
year={2010}, 
volume={chap. \MakeUppercase{\romannumeral 18}},
note={arXiv:0710.2719v2 [math.DG].}
}

@article{7,
author={Gualtieri, M.}, 
title={Generalized complex geometry},
journal={Ann. of Math.},
volume={2 174.1},
year={2011},
pages={75-123},
note={ arXiv:math0401221v1[math.DG]}
}

@article{8, 
author={Hitchin, N.},
title={ Generalized Calabi-Yau manifolds}, 
journal={Mathematical Institute Oxford},
year={2008},
note={ arXiv:math020999v1[math.DG]}
}

@article{9,
author={Nannicini, A.},
title={ Almost complex structures on cotangent bundles and generalized geometry},
journal={ Journal of Geometry and Physics},
year={2010},
pages={1781-1791}
}

@article{10,
author={Nannicini, A.},
title={Generalized geometry of pseudo-Riemannian manifolds and the generalized {$\overline{\partial }$}-operator},
journal={Advances in Geometry},
year={2016},
volume={16 (2)},
pages={165-173}
}

@article{11,
author={Ricciarini, A.},
title={ Alcuni aspetti della Teoria delle Strutture Complesse e Metriche Generalizzate},
journal={ Master's Thesis Università Degli Studi di Firenze },
year={2025}
}

@article{12,
author={Rovenski, V. and Zlatanovic, M.},
title={ Weak metric structures on generalized Riemannian metrics}, 
note={arXiv:2506.23019v2 [math.DG].}
}

\end{document}